\documentclass{amsart}

\newtheorem{theorem}[equation]{Theorem}
\newtheorem{lemma}[equation]{Lemma}
\newtheorem{corollary}[equation]{Corollary}
\newtheorem{proposition}[equation]{Proposition}

\numberwithin{equation}{section}

\usepackage{amscd,amsmath,amssymb,verbatim}

\begin{document}

\title{The Dwork-Frobenius operator on hypergeometric series}
\author{Alan Adolphson}
\address{Department of Mathematics\\
Oklahoma State University\\
Stillwater, Oklahoma 74078}
\email{adolphs@math.okstate.edu}
\author{Steven Sperber}
\address{School of Mathematics\\
University of Minnesota\\
Minneapolis, Minnesota 55455}
\email{sperber@math.umn.edu}
\date{\today}
\begin{abstract}
We describe the action of the Dwork-Frobenius operator on certain $A$-hypergeometric series.  As a consequence, we obtain an integrality result for the coefficients of those series.  This implies an integrality result for classical hypergeometric series.
\end{abstract}
\maketitle

\section{Introduction}

Proving $p$-integrality results for hypergeometric series is a well-known problem.  Many authors have obtained such results, and we reviewed some of their history and applications in the Introduction to \cite{AS3}.  One of the main applications of $p$-integrality is the derivation of $p$-adic analytic formulas for zeros of zeta and $L$-functions of varieties over finite fields.  Such formulas involve $p$-adic analytic continuations of ratios of hypergeometric series.  Dwork originally obtained such results by studying the action of his Frobenius operator on a solution matrix for the Picard-Fuchs equation of a family of varieties.  A necessary condition for the analytic continuation was that the Picard-Fuchs equation have a solution with $p$-integral coefficients.  

We developed a different method for obtaining such results using $A$-hyper\-geometric series \cite{AS1,AS4}.  Our method avoids the problem of computing the Picard-Fuchs equation and finding its solution matrix.  However, we need to know $p$-integrality for not just one hypergeometric series but for an entire class of related hypergeometric series.  In \cite{AS4}, the coefficients of the relevant series were multinomial coefficients, hence trivially $p$-integral, so this difficulty did not arise.  In \cite{AS1}, the relevant series were partial derivatives of a single series, so it sufficed to prove integrality for that one series.  This followed from \cite[Theorem~6.3]{AS}.  

To apply our idea to more general situations, however, we needed an integrality result for all the series in the class.  In the special case where the coefficients of the $A$-hypergeometric series are ratios of factorials, the desired theorem is \cite[Theorem~2.4]{AS3}.  In this article we generalize that result to a wider class of hypergeometric series whose coefficients are ratios of Pochhammer symbols (Theorems~2.21 and~12.15).  We expect that this will lead to ratios of hypergeometric series with $p$-adic analytic continuation that have arithmetic import.

This paper is organized as follows.  In Section 2 we describe the $A$-hypergeometric series we are considering, establish some notation and basic properties, and state our main result.
Section 3 contains a construction of these $A$-hypergeometric series, as well as some related series that satisfy better $p$-adic estimates, by taking products of simpler one-variable series $\xi_j(t)$ and~$\hat{\xi}_j(t)$.  This ``separation of variables'' allows us to reduce some arguments to the one-variable case.  In Section 4, we introduce the $p$-adic spaces on which the Dwork-Frobenius operator will act, and, in Section~5, we show that the $\xi_j(t)$ and $\hat{\xi}_j(t)$ lie in these spaces.  Section 6 is used to construct an auxiliary function, an analogue of the $p$-adic Gamma function, that will play a role in Section~7.  Section 7 applies an idea of Dwork (writing as Boyarsky \cite{D1}): we show that the Dwork-Frobenius operator maps each $\xi_j(t)$ to another $\xi'_j(t)$ times a simple factor.  This result is extended to the $\hat{\xi}_j(t)$ in Section 8.  We extend this result to $A$-hypergeometric series in 
Section~9, showing that $A$-hypergeometric series have an ``eigenvector-like'' property for the Dwork-Frobenius operator (Theorem~9.13).  In Section 10 we restrict the Dwork-Frobenius operator to those $A$-hypergeometric series that potentially have $p$-integral coefficients (Theorem~10.5).  This is the key result, and it allows us to give a simple proof of our main result, Theorem~2.21, in Section~11.  As an example, we derive in Section 12 an integrality result for classical hypergeometric series (Theorem 12.15). 

This paper is completely self-contained except for a reference to Dwork \cite{D} for an estimate for the $p$-divisibility of a Pochhammer symbol in Section~5, a reference to an elementary result from \cite{AS} in Section~2, and some references to $p$-adic estimates from \cite[Section 3]{AS1} in Sections~3 and~6.

We note that the solutions $F_u(\Lambda)$ of the $A$-hypergeometric system with parameter $u$ described in Corollary~2.16 and studied elsewhere in this paper are formal solutions, they do not necessarily belong to a Nilsson ring or converge anywhere.

{\bf Notation.}  It is traditional to express the coefficients of hypergeometric series in terms of Pochhammer symbols.  For $z\in{\mathbb C}$ and $l\in{\mathbb Z}$, one sets $(z)_l = \Gamma(z+l)/\Gamma(z)$.  This is well-defined except when $z+l$ is a nonpositive integer and $z$ is a positive integer, due to the poles of the Gamma function.  Otherwise, the functional equation of the Gamma function gives
\[ (z)_l = \begin{cases} 1 & \text{if $l=0$,} \\ z(z+1)\cdots (z+l-1) & \text{if $l>0$,} \\ \displaystyle{\frac{1}{(z-1)(z-2)\cdots(z+l)}} & \text{if $l<0$ and $z\not\in\{1,2,\dots,-l\}$.} \end{cases} \]
Working with $A$-hypergeometric series, we find it more convenient to use the symbol
\[ [z]_l = \begin{cases} 1 & \text{if $l=0$,} \\ \displaystyle \frac{1}{(z+1)(z+2)\cdots (z+l)} &  \text{if $l>0$ and $z\not\in\{-1,-2,\dots,-l\}$,} \\
z(z-1)\cdots(z+l+1) & \text{if $l<0$.} \end{cases} \]
It has the property that for all $l\in{\mathbb Z}$
\begin{equation}
\frac{d}{dt}\bigg([z]_lt^{z+l}\bigg) = [z]_{l-1}t^{z+l-1},
\end{equation}
hence if $z$ is a $p$-integral rational number for some prime number $p$
\begin{equation}
{\rm ord}_p\: [z]_{l_1}\geq {\rm ord}_p\:[z]_{l_2} \text{ if } l_1\leq l_2.
\end{equation}
One checks that $[z]_l$ is defined if and only if $(-z)_{-l}$ is defined, in which case
\begin{equation}
[z]_l = (-1)^l(-z)_{-l}.
\end{equation}

For $N$-tuples $z=(z_1,\dots,z_N)\in{\mathbb C}^N$ and $l=(l_1,\dots,l_N)\in{\mathbb Z}^N$ we extend this definition by setting
\[ [z]_l = \prod_{j=1}^N [z_j]_{l_j}, \]
assuming always that $z_j\not\in\{-1,-2,\dots,-l_j\}$ if $l_j>0$.  

\section{Statement of results}

We develop some theory that will allow us to state the main result.  We begin by describing the hypergeometric series we are considering.  

Let $A = \{{\bf a}_j\}_{j=1}^N$ with ${\bf a}_j = (a_{1j},a_{2j},\dots,a_{nj})\in{\mathbb Z}^n$ for all $j$.  We assume that the ${\bf a}_j$ lie on a hyperplane $w(u)=1$, where $w(u) = \sum_{i=1}^n w_iu_i$ with $w_i\in{\mathbb Q}$ for $i=1,\dots,n$.  The linear form $w$ defines a function $w:{\mathbb R}^n\to {\mathbb R}$; we refer to $w(u)$ as the {\it weight\/} of $u$.  Note that for $u=\sum_{j=1}^N c_j{\bf a}_j$ we have $w(u) = \sum_{j=1}^N c_j$.

Let $L$ be the lattice of relations on the set $A$:
\[ L = \bigg\{ l=(l_1,\dots,l_N)\in{\mathbb Z}^N\mid \sum_{j=1}^N l_j{\bf a}_j = {\bf 0}\bigg\}. \]
For $l\in L$ define the box operator $\Box_l$ by
\begin{equation}
\Box_l = \prod_{l_j>0} \bigg(\frac{\partial}{\partial\Lambda_j}\bigg)^{l_j} - \prod_{l_j<0} \bigg(\frac{\partial}{\partial \Lambda_j}\bigg)^{-l_j}.
\end{equation}
The Euler operators for a parameter $u=(u_1,\dots,u_n)\in {\mathbb C}^n$ are defined by 
\begin{equation}
Z_i = \sum_{j=1}^N a_{ij}\Lambda_j\frac{\partial}{\partial\Lambda_j} - u_i
\end{equation}
for $i=1,\dots,n$.  The {\it $A$-hypergeometric system with parameter $u$\/} is the system of partial differential equations consisting of the box operators $\Box_l$ for $l\in L$ and the $Z_i$ for $i=1,\dots,n$.

\begin{comment}
Let $z\in{\mathbb C}$.  If $z\not\in{\mathbb Z}_{<0}$ and $l\in{\mathbb Z}$, or if $z\in{\mathbb Z}_{<0}$ and $l\in{\mathbb Z}_{<-z}$, we define
\[ [z]_l = \begin{cases} 1 & \text{if $l=0$,} \\ \displaystyle \frac{1}{(z+1)(z+2)\cdots (z+l)} &  \text{if $l>0$,} \\
z(z-1)\cdots(z+l+1) & \text{if $l<0$.} \end{cases} \]
Although it is more traditional to use the Pochhammer symbol to express the coefficients of hypergeometric series, we prefer the symbol $[z]_l$.  It has the property that
\begin{equation}
\frac{d}{dt}\bigg([z]_lt^{z+l}\bigg) = [z]_{l-1}t^{z+l-1},
\end{equation}
hence if $z$ is a $p$-integral rational number for some prime number $p$
\begin{equation}
{\rm ord}_p\: [z]_{l_1}\geq {\rm ord}_p\:[z]_{l_2} \text{ if } l_1\leq l_2.
\end{equation}
For $N$-tuples $z=(z_1,\dots,z_N)\in{\mathbb C}^N$ and $l=(l_1,\dots,l_N)\in{\mathbb Z}^N$ we extend this definition by setting
\[ [z]_l = \prod_{j=1}^N [z_j]_{l_j}, \]
assuming always that $z_j<-l_j$ if $z_j\in{\mathbb Z}_{<0}$.  
\end{comment}

We describe the conditions on the parameters that will allow us to prove integrality results.  Let ${\mathbb Z}A\subseteq{\mathbb Z}^n$ be the abelian group generated by $A$, let ${\mathbb Q}A\subseteq{\mathbb Q}^n$ be the rational vector space generated by $A$, let ${\mathbb R}A\subseteq{\mathbb R}^n$ be the real vector space generated by $A$, and let $C(A)\subseteq{\mathbb R}^n$ be the real cone generated by $A$, a cone with vertex at the origin.  Let $Q\in{\mathbb Q}A$ and consider the shifted lattice $Q+{\mathbb Z}A$.  

For a closed face $\sigma$ of the negative cone $-C(A)$ we denote by $\sigma^\circ$ its relative interior, i.~e., $\sigma^\circ$ equals $\sigma$ minus all its proper closed subfaces.  Fix a closed face $\sigma$ for which ${\mathcal M}:=(Q+{\mathbb Z}A)\cap\sigma^\circ$ is nonempty.  Then for each $u\in{\mathcal M}$ the closed face $\sigma$ is the smallest closed face of $-C(A)$ that contains $u$.  Our goal is to construct a family of series parameterized by $u\in{\mathcal M}$ which are formal solutions of the $A$-hypergeometric system with parameter $u$ and which have $p$-integral coefficients for all primes $p$ lying in certain residue classes.  The first step is to choose a distinguished element of ${\mathcal M}$.  

Since ${\mathcal M}$ is discrete and the weight function $w$ is nonpositive on $-C(A)$ we can choose an element $\beta\in{\mathcal M}$ for which
\begin{equation}
w(\beta) = \max\{w(u)\mid u\in{\mathcal M}\}.
\end{equation}
We examine some consequences of this condition.  Since $\beta\in{\mathbb Q}A\cap(-C(A))$ we can write 
\begin{equation}
\beta = \sum_{j=1}^N v_j{\bf a}_j
\end{equation}
 for some $v_j\in{\mathbb Q}_{\leq 0}$.  Put $v=(v_1,\dots,v_N)$.  Note that $-{\bf a}_j\in\sigma$ if $v_j<0$.  If $v_j=0$, then ${\bf a}_j$ may or may not lie on $\sigma$.
\begin{lemma}
Suppose that $\beta$ satisfies $(2.3)$.  Then $-1\leq v_j\leq 0$ for $j=1,\dots,N$.  
\end{lemma}

\begin{proof}
Suppose, for example, that $v_1<-1$.  Then
\[ \beta + {\bf a}_1 = (v_1+1){\bf a}_1 + \sum_{j=2}^N v_j{\bf a}_j. \]
The coefficients on the right-hand side are nonpositive and are strictly negative whenever $v_j$ is strictly negative.  This implies that $\beta+{\bf a}_1\in{\mathcal M}$, but $w(\beta+{\bf a}_1)=w(\beta)+1>w(\beta)$.  This contradicts the choice of $\beta$ satisfying~(2.3).
\end{proof}

\begin{lemma}
Suppose that $\beta$ satisfies $(2.3)$.  If $J$ is a proper subset of the set $\{j\in\{1,\dots,N\}\mid v_j=-1\}$, then
\[ -\sum_{j\in J}{\bf a}_j - \sum_{\{j\mid -1<v_j<0\}} {\bf a}_j\not\in\sigma^\circ. \]
\end{lemma}
 
 \begin{proof}
Suppose $-\sum_{j\in J}{\bf a}_j - \sum_{\{j\mid -1<v_j<0\}} {\bf a}_j\in\sigma^\circ$ for some proper subset $J\subseteq
\{j\in\{1,\dots,N\}\mid v_j=-1\}$.  Then
\[ \beta + \sum_{\{j\mid v_j=-1,\,j\not\in J\}} {\bf a}_j =-\sum_{j\in J} {\bf a}_j +\sum_{\{j\mid -1<v_j<0\}} v_j{\bf a}_j\in {\mathcal M}. \]
But $\displaystyle w\bigg(\beta +  \sum_{\{j\mid v_j=-1,\,j\not\in J\}} {\bf a}_j\bigg)>w(\beta)$, contradicting (2.3).
\end{proof}

Let 
\begin{align*}
E&=\{(l_1,\dots,l_N)\in{\mathbb Z}^N\mid \text{$l_j\leq 0$ if $v_j=-1$, $l_j\geq 0$ if $v_j=0$}\}, \\
E_+&=\{(l_1,\dots,l_N)\in{\mathbb Z}^N\mid \text{$l_j\geq 0$ if $v_j=0$}\}.
\end{align*}
Let $u\in{\mathcal M}$ and consider an equation of the form
\begin{equation}
u=\sum_{j=1}^N (v_j+l_j){\bf a}_j\;\text{with $(l_1,\dots,l_N)\in E_+$.}
\end{equation}

\begin{lemma}
In Equation $(2.7)$ we have $l_j=0$ if $-{\bf a}_j\not\in\sigma$.
\end{lemma}

\begin{proof}
Let $H\subseteq{\mathbb R}^n$ be a hyperplane of support for the face $\sigma$.  Then $H$ can be defined by a homogeneous linear equation $h=0$ with $h(\sigma)=0$ and $h(x)>0$ for $x\in\big(-C(A)\big)\setminus\sigma$.  In particular, $h({\bf a}_j)=0$ if $v_j\neq 0$.  If $l_j>0$ for some $j$ with $v_j=0$ and $-{\bf a}_j\not\in\sigma$, then (2.7) implies $h(u)>0$, contradicting the hypothesis that $u\in{\mathcal M}$.
\end{proof}

\begin{proposition}
Suppose that $\beta$ satisfies $(2.3)$ and let $u\in{\mathcal M}$ satisfy an equation of the form~$(2.7)$.  Then $l_j\leq 0$ if $v_j=-1$, i.~e., $(l_1,\dots,l_N)\in E$.
\end{proposition}

\begin{proof}
Rewrite (2.7) as
\begin{multline}
\sum_{\{j\mid v_j=-1,\,l_j\leq 0\}} (-1+l_j){\bf a}_j + \sum_{\{j\mid -1<v_j<0,\,l_j\leq 0\}}(v_j+l_j){\bf a}_j \\ 
=  u-\sum_{\{j\mid -1\leq v_j<0,\,l_j>0\}}(v_j+l_j){\bf a}_j - \sum_{\{j\mid v_j=0\}} l_j{\bf a}_j.
\end{multline}
Since $u\in{\mathcal M}$ and $l_j=0$ if $-{\bf a}_j\not\in\sigma$ by Lemma 2.8, the right-hand side of (2.10) lies in $\sigma^\circ$.  This implies that
\[ \sum_{\{j\mid v_j=-1,\,l_j\leq 0\}} (-1+l_j){\bf a}_j + \sum_{\{j\mid -1<v_j<0,\,l_j\leq 0\}}(v_j+l_j){\bf a}_j \in\sigma^\circ. \]
By Lemma 2.6 this would be a contradiction unless $l_j\leq 0$ for all $j$ with $v_j=-1$.  
\end{proof}

In the special case $u=\beta$, Proposition 2.9 reduces to the assertion that $v$ have minimal negative support, guaranteeing that there is an associated logarithm-free solution of the $A$-hypergeometric system with parameter $\beta$.  We show next that in fact we get logarithm-free solutions for all $u\in{\mathcal M}$ when $\beta$ satisfies (2.3) (Corollary~2.16 below).

For $u\in Q+{\mathbb Z}A$ put
\[ E(u) = \bigg\{ l=(l_1,\dots,l_N)\in E\mid \sum_{j=1}^N (v_j+l_j){\bf a}_j = u\bigg\} \]
where the $v_j$ are as in (2.4) and define
\begin{equation}
F_u(\Lambda_1,\dots,\Lambda_N) = \sum_{l\in E(u)} [v]_l\Lambda^{v+l},
\end{equation}
where we write $\Lambda^{v+l}$ for $\prod_{j=1}^N \Lambda_j^{v_j+l_j}$.  Note that the restriction to $l\in E(u)$ in (2.11) guarantees that $[v]_l$ is well-defined.

Put
\[ E_{\sigma} = \{ l=(l_1,\dots,l_N)\in E\mid\text{$l_j=0$ if $-{\bf a}_j\not\in\sigma$}\} \]
and 
\[ E_{\sigma}(u) = \bigg\{ l=(l_1,\dots,l_N)\in E_{\sigma}\mid \sum_{j=1}^N (v_j+l_j){\bf a}_j = u\bigg\}. \]
For $u\in{\mathcal M}$ we have by Lemma 2.8 the sharper formula
\begin{equation}
F_u(\Lambda_1,\dots,\Lambda_N) = \sum_{l\in E_{\sigma}(u)} [v]_l\Lambda^{v+l},
\end{equation}

\begin{lemma}
Suppose that $\beta$ satisfies $(2.3)$ and let $u\in{\mathcal M}$.  Then
\[ \frac{\partial}{\partial \Lambda_k} F_u(\Lambda) = \begin{cases} F_{u-{\bf a}_k}(\Lambda) & \text{if ${\bf a}_k\in\sigma$,} \\ 0 & \text{if ${\bf a}_k\not\in\sigma$.} \end{cases} \]
\end{lemma}

\begin{proof}
If ${\bf a}_k\not\in\sigma$, then $\Lambda_k$ does not occur to a nonzero power in (2.12), so $\partial F_u/\partial\Lambda_k = 0$.  A straightforward calculation from (2.12) shows that applying $\partial/\partial\Lambda_k$ to a term in $F_u(\Lambda)$ gives either 0 or a term of $F_{u-{\bf a}_k}(\Lambda)$.  The main point of the proof is to show that every monomial in $F_{u-{\bf a}_k}(\Lambda)$ is obtained by applying $\partial/\partial\Lambda_k$ to some monomial in $F_u(\Lambda)$.  

Suppose that ${\bf a}_k\in\sigma$ and consider a monomial $[v]_l\Lambda^{v+l}$ in $F_{u-{\bf a}_k}(\Lambda)$.  Then $l\in E_{\sigma}(u-{\bf a}_k)$ and
\begin{equation}
\sum_{j=1}^N (v_j+l_j){\bf a}_j = u-{\bf a}_k.
\end{equation}
Suppose first that $v_k=-1$.  Then (2.14) gives
\begin{equation}
u = (v_k + l_k + 1){\bf a}_k + \sum_{\substack{j=1\\ j\neq k}}^{N} (v_j+l_j){\bf a}_j. 
\end{equation}
By Proposition 2.9 we have $l_k+1\leq 0$, so if we put 
\[ l'=(l_1,\dots,l_{k-1},l_k+1,l_{k+1},\ldots,l_N) \]
 then $l'\in E_\sigma$ and $\sum_{j=1}^N (v_j + l'_j){\bf a}_j = u$ by (2.15), i.~e., $l'\in E_\sigma( u)$. Thus the monomial $[v]_{l'}\Lambda^{v+l'}$ appears in the series (2.12) and applying $\partial/\partial\Lambda_k$ gives $[v]_l\Lambda^{v+l}$.  Next suppose that $-1<v_k\leq 0$.   If we again define $l'=(l_1,\dots,l_{k-1},l_k+1,l_{k+1},\ldots,l_N)$, then $l'\in E_\sigma(u)$ and the monomial $[v]_{l'}\Lambda^{v+l'}$ appears in the series (2.12).  Applying $\partial/\partial\Lambda_k$ gives $[v]_l\Lambda^{v+l}$.
\end{proof}

\begin{corollary}
Suppose that $\beta$ satisfies $(2.3)$ and $u\in{\mathcal M}$.  Then $F_u(\Lambda)$ satisfies the $A$-hypergeometric system with parameter $u$.
\end{corollary}
 
 \begin{proof}
It follows from the condition $\sum_{j=1}^N (v_j+l_j){\bf a}_j = u$ on the sum in (2.12) that each monomial $\Lambda^{v+l}$ satisfies the operators $Z_i$ for the parameter $u$.  Let $l=(l_1,\dots,l_N)\in L$.  Then
\begin{equation}
\sum_{l_j>0} l_j{\bf a}_j = -\sum_{l_j<0} l_j{\bf a}_j.
\end{equation}
This implies that $l_j>0$ for some ${\bf a}_j\not\in\sigma$ if and only if $l_{j'}<0$ for some ${\bf a}_{j'}\not\in\sigma$.  In this case, by Lemma~2.13,
\[ \prod_{l_j>0} \bigg(\frac{\partial}{\partial\Lambda_j}\bigg)^{l_j}F_u(\Lambda) =  \prod_{l_j<0} \bigg(\frac{\partial}{\partial\Lambda_j}\bigg)^{-l_j}F_u(\Lambda) = 0 \]
so $\Box_l F_u(\Lambda) = 0$.  Otherwise, Lemma 2.13 implies that
\[ \prod_{l_j>0} \bigg(\frac{\partial}{\partial\Lambda_j}\bigg)^{l_j}F_u(\Lambda) = F_{u-\sum_{l_j>0} l_j{\bf a}_j}(\Lambda) \]
and
\[ \prod_{l_j<0} \bigg(\frac{\partial}{\partial\Lambda_j}\bigg)^{-l_j}F_u(\Lambda) = F_{u+\sum_{l_j<0} l_j{\bf a}_j}(\Lambda). \]
These two expressions are equal by (2.17), so again $\Box_l F_u(\Lambda) = 0$.  
\end{proof}

Let $v=(v_1,\dots,v_N)$ be as in (2.4).  Choose a positive integer $D$ such that $Dv_j\in{\mathbb Z}$ for all $j$ and choose a positive integer $h$ such that $(h,D)=1$.  For each $j$ there exists a unique rational number $v'_j$, $-1\leq v'_j\leq 0$, with $Dv'_j\in{\mathbb Z}$ and $hv_j-v'_j\in \{0,-1,\dots,-(h-1)\}$.  Furthermore, $v'_j$ depends only on the residue class of $h\pmod{D}$: if $h\equiv \tilde{h}\pmod{D}$ then both $h$ and $\tilde{h}$ define the same map $v_j\to v'_j$ (see \cite[Section~3]{AS}).  We denote the $i$-fold iteration of this map by $v_j\to v_j^{(i)}$, i.~e., $v_j^{(i)} = (v_j^{(i-1)})'$.  Since $(h,D)=1$, there exists a positive integer $a$ such that for all $j$ one has $v_j^{(a)} = v_j$.  We set $v^{(i)} = (v_1^{(i)},\dots,v_N^{(i)})$.  

For each $i$ we define $\beta^{(i)} = \sum_{j=1}^N v_j^{(i)}{\bf a}_j$.  Note that $\beta^{(i)}\in h^iQ + {\mathbb Z}A$.  Note also that if $v_j=-1$ then $v_j'=-1$ and if $v_j=0$ then $v_j'=0$.  If $-1<v_j<0$, then $-1<v_j'<0$.  It follows that $\sigma$, the smallest closed face of $-C(A)$ containing $\beta$, is also the smallest closed face of $-C(A)$ containing $\beta^{(i)}$ for all $i$.  Let ${\mathcal M}^{(i)} = (h^iQ + {\mathbb Z}A)\cap\sigma^\circ$.  

For $u^{(i)}\in h^iQ+{\mathbb Z}A$ we define as in (2.11)
\begin{equation}
F_{u^{(i)}}(\Lambda_1,\dots,\Lambda_N) = \sum_{l\in E(u^{(i)})}[v^{(i)}]_l\Lambda^{v^{(i)}+l}.
\end{equation}
And as in (2.12) the sharper formula
\begin{equation}
F_{u^{(i)}}(\Lambda_1,\dots,\Lambda_N) = \sum_{l\in E_\sigma(u^{(i)})}[v^{(i)}]_l\Lambda^{v^{(i)}+l}
\end{equation}
holds for $u^{(i)}\in {\mathcal M}^{(i)}$.  

Consider the condition analogous to (2.3):
\begin{equation}
w(\beta^{(i)}) = \max\{w(u^{(i)})\mid u^{(i)}\in{\mathcal M}^{(i)}\}.
\end{equation}
By Corollary 2.16, this condition implies that for all $u^{(i)}\in{\mathcal M}^{(i)}$ the series $F_{u^{(i)}}(\Lambda)$ satisfies the $A$-hypergeometric system with parameter $u^{(i)}$.

The following theorem is the main result of this paper.
\begin{theorem}
Suppose that $(2.20)$ holds for $i=0,1,\dots,a-1$.  Then for $i=0,1,\dots,a-1$ and all $u^{(i)}\in{\mathcal M}^{(i)}$, the hypergeometric series $F_{u^{(i)}}(\Lambda)$ has $p$-integral coefficients for all primes $p\equiv h\pmod{D}$.
\end{theorem}

When $h\equiv 1\pmod{D}$ we have $a=1$, so Theorem 2.21 gives the following corollary.
\begin{corollary}
If $\beta$ satisfies $(2.3)$, then for all $u\in{\mathcal M}$ the hypergeometric series $F_u(\Lambda)$ has $p$-integral coefficients for all primes $p\equiv 1\pmod{D}$.
\end{corollary}

{\bf Remark.}  By Lemma 2.8 one has $l_j=0$ in (2.7) for all $u\in{\mathcal M}$ if $-{\bf a}_j\not\in\sigma$ (in which case $v_j=0$), thus the corresponding variable $\Lambda_j$ does not appear in the series $F_u(\Lambda)$.  One can therefore replace the original set $A$ by the set $A':=\{{\bf a}_j\mid {\bf a}_j\in\sigma\}$, and one then has ${\mathcal M} = (Q+{\mathbb Z}A')\cap(-C(A')^\circ)$.  It thus suffices to prove Theorem~2.21 in the case where $\sigma = -C(A)$ and ${\mathcal M} = (Q+{\mathbb Z}A)\cap (-C(A)^\circ)$.  

The main application we give, in Section 12, is a condition for the $p$-integrality of classical hypergeometric series of the form
\begin{equation}
\sum_{s_1,\dots,s_m=0}^\infty (\theta_1)_{C_1(s)}\cdots(\theta_n)_{C_n(s)}\frac{t_1^{s_1}\cdots t_m^{s_m}}{s_1!\cdots s_m!},
\end{equation}
where $C_1,\dots,C_n$ are homogeneous linear forms with integral coefficients and $\theta_1,\dots,\theta_n$ are $p$-integral rational numbers in the interval $[0,1]$.  

\section{Generating series}

The proof of Theorem 2.21 will require several steps.  We begin with some notation.  Fix for the remainder of this article the vector $Q\in{\mathbb Q}A$, a vector $\beta\in (Q+{\mathbb Z}A)\cap(-C(A)^\circ)$, and a rational vector $v=(v_1,\dots,v_N)\in{\mathbb Q}^N$ whose coordinates lie in the interval $[-1,0]$ and such that $\beta = \sum_{j=1}^N v_j{\bf a}_j$.  We thus have ${\mathcal M} = (Q + {\mathbb Z}A)\cap (-C(A)^\circ)$ and by (2.11) 
\begin{equation}
F_u(\Lambda) = \sum_{l\in E(u)}[v]_l\Lambda^{v+l}
\end{equation}
for $u\in Q+{\mathbb Z}A$.  We do not assume (2.3) or (2.20) until Section~10.  

We also fix a prime number $p$ for which all $v_j$ are $p$-integral.  Let ${\mathbb Z}_p$ be the $p$-adic integers, and ${\mathbb Q}_p$ the field of $p$-adic numbers.  We denote by ${\mathbb C}_p$ the completion of an algebraic closure of~${\mathbb Q}_p$.  The norm on ${\mathbb C}_p$ is denoted by $\lvert\cdot\rvert$ and is normalized by the condition $\lvert p\rvert = 1/p$.  The corresponding additive valuation is denoted ${\rm ord}$ and satisfies ${\rm ord}\:p = 1$.  We denote by ${\mathbb N}$ the nonnegative integers.

It will be useful to have a generating series construction for the $F_u(\Lambda)$.  For $-1<v_j\leq 0$, define functions of one variable
\begin{equation}
 f_j(t) = \sum_{l=-\infty}^\infty [v_j]_lt^{v_j+l}. 
 \end{equation}
Note that in the special case $v_j=0$ this simplifies to $f_j(t) = \exp t$.  For $v_j=-1$, we define
\begin{equation}
 f_j(t) = \sum_{l=-\infty}^{0} [v_j]_lt^{v_j+l} = \sum_{l=0}^\infty (-1)^l l!\, t^{-1-l}. 
 \end{equation}
The following proposition is a straightforward calculation from (3.1).

\begin{proposition}
We have
\[ \prod_{j=1}^N f_j(\Lambda_jx^{{\bf a}_j}) = \sum_{u\in Q + {\mathbb Z}A} F_u(\Lambda)x^u. \]
\end{proposition}

We need to modify this formula to obtain related series with equivalent integrality properties that are easier to study $p$-adically.
Let $\pi_0\in{\mathbb C}_p$ be a root of the equation $\sum_{i=0}^\infty t^{p^i}/{p^i}=0$ satisfying ${\rm ord}\:\pi_0 = 1/(p-1)$.  
We set
\begin{equation} \xi_j(t) = f_j(\pi_0 t) = \begin{cases} \displaystyle \sum_{l=-\infty}^\infty [v_j]_l(\pi_0 t)^{v_j+l} & \text{if $v_j>-1$,} \\ \displaystyle \sum_{l=0}^\infty (-1)^ll!(\pi_0 t)^{-1-l} & \text{if $v_j=-1$,} \end{cases} \end{equation}
so from Proposition 3.4 we have
 \begin{equation}
 \prod_{j=1}^N \xi_j(\Lambda_jx^{{\bf a}_j}) = \sum_{u\in Q + {\mathbb Z}A} F_u(\Lambda)\pi_0^{w(u)}x^u. 
 \end{equation}
The $\xi_j(t)$ do not lead to sufficiently strong $p$-adic estimates, so we introduce some related series that lead to better estimates.  

Let ${\rm AH}(t) = \exp(\sum_{i=0}^\infty t^{p^i}/p^i)$ be the Artin-Hasse series, a power series in $t$ with $p$-integral coefficients.  Put $\theta(t) = {\rm AH}(\pi_0 t)$.  If we write $\theta(t) = \sum_{i=0}^\infty \theta_it^i$, then
\begin{equation}
{\rm ord}\:\theta_i\geq i/(p-1).
\end{equation}
We define $\hat{\theta}(t) = \prod_{j=0}^\infty \theta(t^{p^j})$, which gives $\theta(t) = \hat{\theta}(t)/\hat{\theta}(t^p)$.  We write $\hat{\theta}(t) = \sum_{i=0}^\infty \hat{\theta}_i(\pi_0t)^i/i!$, and by \cite[Equation~(3.8)]{AS1} we have
\begin{equation}
{\rm ord}\:\hat{\theta}_i\geq 0.
\end{equation}
We also need the series $\hat{\theta}_1(t):= \hat{\theta}(t)/\exp(\pi_0t)$.  We write $\hat{\theta}_1(t) = \sum_{i=0}^\infty \hat{\theta}_{1,i}(\pi_0t)^i/{i!}$ and have by \cite[Equation~(3.10)]{AS1}
\begin{equation}
{\rm ord}\:\hat{\theta}_{1,i}\geq \frac{i(p-1)}{p}. 
\end{equation}

Suppose first that $v_j>-1$.  Define
\begin{equation}
\hat{\xi}_j(t) = \xi_j(t)\hat{\theta}_1(t) 
\end{equation}
 and write
 \[ \hat{\xi}_j(t) = \sum_{l=-\infty}^\infty g(v_j,l)(\pi_0 t)^{v_j+l}. \]
From (3.5) we have
 \[ g(v_j,l) = \sum_{i=0}^\infty [v_j]_{l-i}\frac{\hat{\theta}_{1,i}}{i!}, \]
 hence
 \begin{equation}
  \frac{g(v_j,l)}{[v_j]_l} = \sum_{i=0}^\infty \frac{[v_j]_{l-i}}{[v_j]_li!}\hat{\theta}_{1,i}. 
  \end{equation}
 One checks that for all $l\in{\mathbb Z}$ and all $i\in{\mathbb Z}_{\geq 0}$ one has
 \[ \frac{[v_j]_{l-i}}{[v_j]_li!} = \frac{(v_j+l)(v_j+l-1)\cdots(v_j+l-i+1)}{i!} = \binom{v_j+l}{i}, \]
 which is $p$-integral since $v_j$ is $p$-integral.  Since the $i=0$ term of the series in (3.11) equals 1, Equation (3.9) implies that $g(v_j,l)/[v_j]_l$ is $p$-integral and
 \begin{equation}
 \frac{g(v_j,l)}{[v_j]_l}\equiv 1\pmod{\pi_0}.
 \end{equation}
 In particular,
 \begin{equation}
 {\rm ord}\:g(v_j,l) = {\rm ord}\:[v_j]_l.
 \end{equation}
 
Let $\gamma$ be the operator on power series defined by
 \[ \gamma\bigg(\sum_{l=-\infty}^\infty c_lt^l\bigg) = \sum_{l=-\infty}^{-1} c_lt^l. \]
For $v_j=-1$, define
\begin{equation}
 \hat{\xi}_j(t) =\gamma\big(\xi_j(t)\hat{\theta}_1(t)\big). 
 \end{equation}
We write
\[ \hat{\xi}_j(t) =\sum_{l=0}^\infty g(-1,l)(\pi_0t)^{-1-l} \]
where by (3.5)
\[ g(-1,l) = \sum_{i=0}^\infty (-1)^{l+i} (l+i)! \frac{\hat{\theta}_{1,i}}{i!}. \]
We thus have
\[ \frac{g(-1,l)}{l!} = \sum_{i=0}^\infty (-1)^{l+i} \binom{l+i}{i}\hat{\theta}_{1,i}. \]
The $i=0$ term of this series is $(-1)^l$, so Equation (3.9) implies that $g(-1,l)/l!$ is $p$-integral and 
\begin{equation}
\frac{g(-1,l)}{l!}\equiv (-1)^l\pmod{\pi_0}.
\end{equation}
In particular,
\begin{equation}
{\rm ord}\:g(-1,l) = {\rm ord}\: l!.
\end{equation}

For $u\in Q+{\mathbb Z}A$ define
\begin{equation}
 G_u(\Lambda_1,\dots,\Lambda_N) = \sum_{l=(l_1,\dots,l_N)\in E(u)} \bigg(\prod_{j=1}^N g(v_j,l_j)\bigg)\Lambda^{v+l}.
 \end{equation}
 A straightforward calculation gives the analogue of (3.6):
 \begin{equation}
 \prod_{j=1}^N \hat{\xi}_j(\Lambda_jx^{{\bf a}_j}) = \sum_{u\in Q+{\mathbb Z}A} G_u(\Lambda)\pi_0^{w(u)}x^u.
 \end{equation}
 From (3.13) and (3.16) we have immediately the following result.
 \begin{proposition}
 For $u\in Q + {\mathbb Z}A$ the series $F_u(\Lambda)$ has $p$-integral coefficients if and only if the series $G_u(\Lambda)$ has $p$-integral coeffcients.
 \end{proposition}
 
 The point of this proposition will be that it is easier to find good estimates for the coefficients of the $G_u(\Lambda)$ than for the coefficients of the $F_u(\Lambda)$.  As the next step in this process, we define in the next section spaces that contain the $\xi_j(t)$ and $\hat{\xi}_j(t)$.  In Sections 7 and 8 we define Dwork-Frobenius operators on these spaces and determine their action on 
 $\xi_j(t)$ and $\hat{\xi}_j(t)$.

 \section{Some $p$-adic vector spaces}

We construct some vector spaces over ${\mathbb C}_p$ on which the Dwork-Frobenius operator will act.  Consider a collection of elements $c_l\in{\mathbb C}_p$ indexed by $l=(l_1,\dots,l_N)\in{\mathbb Z}^N$.  Let $\lvert\cdot\rvert_\infty$ denote the usual absolute value on real numbers.  We define $\lvert l\rvert_\infty = \sum_{j=1}^N \lvert l_j\rvert_\infty$.  For $\delta>0$, we say that $\lvert c_l\rvert\delta^{\lvert l\rvert_\infty}$ {\it converges to $0$},  
$\lvert c_l\rvert\delta^{\lvert l\rvert_\infty}\to 0$, if for every $\epsilon >0$ we have $\lvert c_l\rvert\delta^{\lvert l\rvert_\infty}<\epsilon$ for all but finitely many $l$.  Put
\begin{equation}
C\{t_1,\dots,t_N\}=\bigg\{\xi(t)  = \sum_{l=(l_1,\dots,l_N)\in{\mathbb Z}^N} c_lt_1^{l_1}\cdots t_N^{l_N}\mid \text{$\lvert c_l\rvert\delta^{\lvert l\rvert_\infty}\to 0$ for all $\delta<1$}\bigg\}.
\end{equation}
For any subset $S\subseteq{\mathbb Z}^N$ we set
\begin{equation}
C_S\{t_1,\dots,t_N\}=\bigg\{\xi(t)  = \sum_{l=(l_1,\dots,l_N)\in S} c_lt_1^{l_1}\cdots t_N^{l_N}\mid \xi(t)\in C\{t_1,\dots,t_N\}\bigg\}.
\end{equation}
The set $C\{t_1,\dots,t_N\}$ is not a space of functions.  These series need not converge at any point $t=t_0\in{\mathbb C}_p$.

Let
\[ H(t) = \sum_{m=(m_1,\dots,m_N)\in{\mathbb N}^N} d_m t_1^{m_1}\cdots t_N^{m_N} \]
be a power series converging on a polydisk 
\[ \{(t_1,\dots,t_N)\in{\mathbb C}_p^N\mid\text{$\lvert t_j\rvert\leq R$ for $j=1,\dots,N$}\} \]
for some $R>1$, i.~e.,  
\begin{equation}
\lvert d_m\rvert R^{\sum_{j=1}^N m_j}\to 0.
\end{equation}
\begin{proposition}
Multiplication by $H(t)$ maps $C\{t_1,\dots,t_N\}$ into itself.
\end{proposition}

\begin{proof}
For $\xi(t) = \sum_{l\in {\mathbb Z}^N} c_lt^l\in C\{t_1,\dots,t_N\}$ we have formally
\begin{equation}
H(t)\xi(t) = \sum_{k\in{\mathbb Z}^N} e_kt^k,
\end{equation}
where
\begin{equation}
e_k = \sum_{m\in{\mathbb N}^N}d_mc_{k-m}.
\end{equation}
We need to show that the series (4.6) converges and that for each $\delta$, $0<\delta<1$, and $\epsilon>0$ one has
\begin{equation}
\lvert e_k\rvert \delta^{\lvert k\rvert_\infty}<\epsilon
\end{equation}
for all but finitely many $k$.  (By ``convergence'' of the series (4.6), we mean that the partial sums $\sum_{m\in [0,M]^N} d_mc_{k-m}$ approach a limit as $M\to\infty$.  This is equivalent to requiring that for every $\epsilon>0$, all but finitely many terms of the series~(4.6) are less than $\epsilon$.)

Choose $r$, $1<r<R$.  We have
\begin{equation}
 \lvert d_mc_{k-m}\rvert r^{\sum_{j=1}^N k_j} = (\lvert d_m\rvert r^{\sum_{j=1}^N m_j})(\lvert c_{k-m}\rvert r^{\sum_{j=1}^N (k_j-m_j)}). 
 \end{equation}
For all but finitely many terms in the sum (4.6) we have $k_j-m_j\leq 0$ for $j=1,\dots,N$.  We then have from (4.8)
\begin{equation}
\lvert d_mc_{k-m}\rvert r^{\sum_{j=1}^N k_j} = (\lvert d_m\rvert r^{\sum_{j=1}^N m_j})(\lvert c_{k-m}\rvert(r^{-1})^{\lvert k-m\rvert_\infty})
\end{equation}
for all but finitely many $m$.  The convergence of (4.6) now follows from (4.1) and~(4.3).

Choose $\delta$, $1/R<\delta<1$.  From (4.6) we have 
\[ \lvert e_k\rvert\delta^{\lvert k\rvert_\infty}\leq \sup_{m\in{\mathbb N}^N} (\lvert d_m\rvert\delta^{-\sum_{j=1}^N m_j})(\lvert c_{k-m}\rvert\delta^{\lvert k\rvert_\infty+\sum_{j=1}^N m_j}). \]
Since $\lvert k\rvert_\infty + \sum_{j=1}^N m_j\geq \lvert k-m\rvert_\infty$ and $\delta<1$, we get
\[ \lvert e_k\rvert\delta^{\lvert k\rvert_\infty}\leq \sup_{m\in{\mathbb N}^N} (\lvert d_m\rvert\delta^{-\sum_{j=1}^N m_j})(\lvert c_{k-m}\rvert\delta^{\lvert k-m\rvert_\infty}). \]
Since $\delta^{-1}<R$, condition (4.3) implies that the set $\{\lvert d_m\rvert\delta^{-\sum_{j=1}^N m_j}\mid m\in{\mathbb N}^N\}$ is bounded above, say
\begin{equation}
 \lvert d_m\rvert\delta^{-\sum_{j=1}^N m_j}\leq M_1\;\text{for all $m\in{\mathbb N}^N$.} 
 \end{equation}
Since $\xi(t)\in C\{t_1,\dots,t_N\}$, the set $\{\lvert c_{k-m}\rvert\delta^{\lvert k-m\rvert_\infty}\mid k\in{\mathbb Z}^N\}$ is bounded above, say
\begin{equation}
 \lvert c_{k-m}\rvert\delta^{\lvert k-m\rvert_\infty}\leq M_2\;\text{for all $k\in{\mathbb Z}^N$.} 
 \end{equation}
Let $\epsilon>0$.  By (4.3) there exists a finite set $S_1$ such that
\begin{equation}
 \lvert d_m\rvert\delta^{-\sum_{j=1}^N m_j}< \frac{\epsilon}{M_2}\;\text{for $m\not\in S_1$.} 
 \end{equation}
Since $\xi(t)\in C\{t_1,\dots,t_N\}$, there exists a  finite set $S_2$ such that
\begin{equation}
 \lvert c_{k-m}\rvert\delta^{\lvert k-m\rvert_\infty}<\frac{\epsilon}{M_1}\;\text{for $k-m\not\in S_2$.} 
 \end{equation}
Let $S_1+S_2=\{s_1+s_2\mid s_1\in S_1\text{ and } s_2\in S_2\}$, a finite set.  Suppose that $k\not\in S_1+S_2$.  If $k-m\not\in S_2$, then (4.10) and (4.13) imply that
\begin{equation}
 (\lvert d_m\rvert\delta^{-\sum_{j=1}^N m_j})(\lvert c_{k-m}\rvert\delta^{\lvert k-m\rvert_\infty})<\epsilon. 
 \end{equation}
If $k-m\in S_2$, i.~e., $k\in m+S_2$, then $m\not\in S_1$ because $k\not\in S_1+S_2$.  Equations (4.11) and (4.12) then imply that (4.14) holds.
\end{proof}

The space $C\{t_1,\dots,t_N\}$ is not a ring, but for pairs involving disjoint sets of variables multiplication is defined.  The following assertion is straightforward to check.
\begin{lemma}
If $\eta_j(t_j) = \sum_{l_j=-\infty}^{\infty} c_{l_j} t_j^{l_j}\in C\{t_j\}$ for $j=1,\dots,N$, then
\[ \prod_{j=1}^N\eta_j(t_j) = \sum_{l=(l_1,\dots,l_N)\in{\mathbb Z}^N}\bigg(\prod_{j=1}^N c_{l_j}\bigg) \bigg(\prod_{j=1}^N t_j^{l_j}\bigg) \in C\{t_1,\dots,t_N\}. \]
\end{lemma}

We also have a commutativity relation for multiplication involving disjoint sets of variables and the multiplication of Proposition~4.4.
\begin{lemma}
Suppose that for $j=1,\dots,N$ the series $H_j(t_j) = \sum_{m_j=0}^\infty d_{m_j}t_j^{m_j}$ converge for $\lvert t_j\rvert\leq R$, $R>1$, and $\eta_j(t_j) = \sum_{l_j=-\infty}^{\infty} c_{l_j} t_j^{l_j}\in C\{t_j\}$.  Then
\begin{equation}
\bigg(\prod_{j=1}^N \big(H_j(t_j)\eta_j(t_j)\big)\bigg) = \bigg(\prod_{j=1}^N H_j(t_j)\bigg)\bigg(\prod_{j=1}^N \eta_j(t_j)\bigg) \in C\{t_1,\dots,t_N\}.
\end{equation}
\end{lemma}

\begin{proof}
The products in (4.17) are well-defined by Proposition 4.4 and Lemma 4.15.   The proof of (4.17) is then a straightforward calculation.
\end{proof}

\section{$p$-adic estimates for $\xi_j(t)$ and $\hat{\xi}_j(t)$}

\begin{proposition}
We have $\xi_j(t), \hat{\xi}_j(t)\in t^{v_j}C\{t\}$.
\end{proposition}

When $v_j=-1$, this result is trivial by (3.5) and (3.16).  When $-1<v_j\leq 0$, it is an immediate consequence of (3.5), (3.13), and the following lemma.  
\begin{lemma}
Let $z$ be a $p$-integral rational number which is not a negative integer.  Then for all $l\in{\mathbb Z}$
\[ {\rm ord}\: \pi_0^l[z]_l\geq -\log_p(p\lvert l\rvert_\infty). \]
\end{lemma}

The remainder of this section is the proof of Lemma 5.2.
For a positive integer~$b$ one checks that ${\rm ord}\: (b)_l\geq (l-s_l)/(p-1)$, where $s_l$ denotes the sum of the digits in the $p$-adic expansion of $l$.  As a function of $b$, $(b)_l$ is continuous in the $p$-adic topology on ${\mathbb Z}_p$, so this estimate holds on ${\mathbb Z}_p$.  If $l$ is a negative integer, it then follows from (1.3) that for $z$ a $p$-integral rational number one has
\begin{equation}
{\rm ord}\: [z]_l\geq \frac{-l-s_{-l}}{p-1}.
\end{equation}
Using the trivial estimate $s_{-l}/(p-1)\leq \log_p(-pl)$ we have finally
\begin{equation}
{\rm ord}\: [z]_l\geq \frac{-l}{p-1} -\log_p(-pl)
\end{equation}
when $l$ is a negative integer.  Estimate (5.4) implies Lemma~5.2 when $l$ is a negative integer.  

 Suppose now that $l>0$ and $z$ is a $p$-integral rational number.  Since $[z]_l = 1/(z+1)_l$, finding a lower bound for ${\rm ord}\:[z]_l$ is equivalent to finding an upper bound for ${\rm ord}\: (z+1)_l$.  
For this we use Dwork \cite[Equation (1.3)]{D}.  Write 
\[ z+1 = -\sum_{i=0}^\infty \sigma_ip^i\;\text{with $0\leq\sigma_i\leq p-1$ for all $i$} \]
and set
\[ \phi(z+1) =-\sum_{i=0}^\infty \sigma_{i+1}p^i. \]
Then
\[ p\phi(z+1)-(z+1) = \sigma_0. \]
Write $l=\sum_{i=0}^{k-1} l_ip^i$ and put $l^{(m)} = l_m+l_{m+1}p + \cdots + l_{k-1}p^{k-1-m}$.  For a real number $x$ define $\rho(x)=0$ if $x\leq 0$ and define $\rho(x)=1$ if $x>0$.  From \cite[Equation~(1.3)]{D} we have
\begin{multline}
{\rm ord}\: (\phi^{m}(z+1))_{l^{(m)}} - {\rm ord}\: (\phi^{m+1}(z+1))_{l^{(m+1)}} \\ 
= l^{(m+1)} + (1+{\rm ord}\:(l^{(m+1)} + \phi^{m+1}(z+1)))\rho(l_m-\sigma_m).
\end{multline}
Summing (5.5) over $m=0,1,\dots,k-1$ gives
\begin{equation}
{\rm ord}\: (z+1)_l = \frac{l-s_l}{p-1} + \sum_{m=0}^{k-1}(1+{\rm ord}\:(l^{(m+1)} + \phi^{m+1}(z+1)))\rho(l_m-\sigma_m).
\end{equation}
For a given $m\in\{0,\dots,k-1\}$ the contribution to the sum on the right-hand side is zero unless $l_m>\sigma_m$.  
When $l_m>\sigma_m$ the contribution is 
\[ 1+{\rm ord}\:(l^{(m+1)} + \phi^{m+1}(z+1)). \]
We have 
\begin{align*} 
 l^{(m+1)} &= l_{m+1} + l_{m+2}p + \cdots + l_{k-1}p^{k-1-m}, \\
 \phi^{m+1}(z+1) &= -\sigma_{m+1} -\sigma_{m+2}p -\dots. 
 \end{align*}
It follows that ${\rm ord}\: (l^{(m+1)} + \phi^{m+1}(z+1))\geq s$ if and only if
\[ l_r=\sigma_r\;\text{for $r=m+1,m+2,\dots,m+s$.} \]

Call a collection of consecutive terms $(l_m,l_{m+1},\dots,l_{m+s})$ of $(l_0,\dots,l_{k-1})$ {\it good\/} if $l_m>\sigma_m$, $l_r=\sigma_r$ for $r=m+1,\dots,m+s$, and $l_{m+s+1}\neq \sigma_{m+s+1}$.  Clearly any two good collections are disjoint.  If we let $\ell$ be the total number of terms of $(l_0,\dots,l_{k-1})$ that appear in some good collection of consecutive terms, then
\begin{equation}
{\rm ord}\:(z+l)_l = \frac{l-s_l}{p-1} + \ell,
\end{equation}
hence
\begin{equation}
{\rm ord}\:[z]_l =  -\frac{l-s_l}{p-1} - \ell
\end{equation}
when $l$ is a positive integer.
In particular, the disjointness of good subsequences implies the estimate
\begin{equation}
{\rm ord}\:[z]_l\geq -\frac{l-s_l}{p-1} - k.
\end{equation}
Since $k\leq \log_p(pl)$ we get finallly
\begin{equation}
{\rm ord}\:[z]_l \geq -\frac{l}{p-1} - \log_p(pl)
\end{equation}
when $l$ is a positive integer.  Estimate (5.10) implies Lemma~5.2 when $l$ is a positive integer.

\section{An auxiliary function}

The purpose of this section is to introduce a function that will appear in formulas for the Dwork-Frobenius action.
We use the following variation of the Dwork exponential function:
\[ \sigma(t):= \exp(\pi_0 t-\pi_0 t^p) =  \sum_{i=0}^\infty \sigma_it^i. \]
From the definition of $\hat{\theta}_1(t)$ we have $\exp(\pi_0 t) = \hat{\theta}(t)/\hat{\theta}_1(t)$, so 
\[ \exp(\pi_0 t - \pi_0 t^p) = \frac{\hat{\theta}(t)\hat{\theta}_1(t^p)}{\hat{\theta}(t^p)\hat{\theta}_1(t)} = \theta(t) \frac{\hat{\theta}_1(t^p)}{\hat{\theta}_1(t)}. \]
By (3.7) and (3.9) the $\sigma_i$ in the series expansion satisfy
\begin{equation}
 {\rm ord}\:\sigma_i\geq \frac{i(p-1)}{p^2}. 
\end{equation}

Consider the series
\[ H(z) = \sum_{l=0}^\infty [-z]_{-l}\sigma_{pl}\pi_0^{-l} = \sum_{l=0}^\infty (-1)^l (z)_l \sigma_{pl}\pi_0^{-l}. \]
The estimate (6.1) implies that this series is an analytic function of $z$ for ${\rm ord}\:z>-(p-1)/p+1/(p-1)$.  The set of analyticity includes ${\mathbb Z}_p$ for $p\geq 3$ but not for $p=2$.  However, since $(z)_l$ is a continuous function on ${\mathbb Z}_p$, the estimates of (6.1) and Lemma~5.2 show that the series $H(z)$ converges on ${\mathbb Z}_p$ to a continuous function for $p\geq 2$.  

\begin{lemma}
The function $H(z)$ assumes unit values for $z\in{\mathbb Z}_p$.
\end{lemma}

\begin{proof}
It follows from the definition that
\[ \sigma_{pl} = (-\pi_0)^l\sum_{i=0}^l \frac{(-1)^i\pi_0^{(p-1)i}}{(pi)!(l-i)!}. \]
This gives
\[ H(z) = \sum_{l=0}^\infty (z)_l\sum_{i=0}^l \frac{(-1)^i\pi_0^{(p-1)i}}{(pi)!(l-i)!}. \]
We also have $\displaystyle (z)_l = (-1)^ll!\binom{-z}{l}$, so
\[ H(z) =  \sum_{l=0}^\infty (-1)^ll!\binom{-z}{l}\sum_{i=0}^l \frac{(-1)^i\pi_0^{(p-1)i}}{(pi)!(l-i)!}. \]
Let $r$ be a positive integer.  Then
\[ H(-r) = \sum_{l=0}^r (-1)^l\binom{r}{l}\sum_{i=0}^l \binom{l}{i}\frac{(-1)^ii!\pi_0^{(p-1)i}}{(pi)!}. \]
Using the relation $\binom{r}{l}\binom{l}{i} = \binom{r}{i}\binom{r-i}{l-i}$ this becomes
\[ H(-r) = \sum_{l=0}^r \sum_{i=0}^l (-1)^{l+i} \binom{r}{i}\binom{r-i}{l-i}\frac{i! \pi_0^{(p-1)i}}{(pi)!}. \]
Interchange the order of summation:
\begin{equation}
H(-r) = \sum_{i=0}^r (-1)^i\binom{r}{i} \frac{i! \pi_0^{(p-1)i}}{(pi)!}\sum_{l=i}^r (-1)^l\binom{r-i}{l-i}.
\end{equation}
The inner sum in (6.3) collapses:
\[ \sum_{l=i}^r (-1)^l \binom{r-i}{l-i} = \begin{cases} (-1)^r & \text{if $i=r$,} \\ 0 & \text{if $i<r$,} \end{cases} \]
so (6.3) reduces to
\begin{equation}
H(-r) =  \frac{r! \pi_0^{(p-1)r}}{(pr)!}.
\end{equation}
From \cite[unnumbered equation between Equations (3.23) and (3.24)]{AS1} we have $\pi_0^{p-1}= -p + up^p$, where ${\rm ord}\:u\geq 0$, so
\begin{equation}
 H(-r) = \frac{r! (-p + up^p)^r}{(pr)!} = \frac{r!\sum_{i=0}^r \binom{r}{i}(-p)^{r-i}(up^p)^i}{(pr)!}. 
 \end{equation}
 
For $i=0$, the contribution to the summation on the right-hand side of (6.5) is $r!(-p)^r/(pr)!$, which is easily checked to be a $p$-adic unit.  For $1\leq i\leq r$, the contribution to the summation on the right-hand side of (6.5) is
\begin{equation}
 \frac{r!\binom{r}{i}(-p)^{r-i}(up^p)^i}{(pr)!}. 
 \end{equation}
The $p$-ordinal of this numerator is greater than or equal to
\[ {\rm ord}\:r!p^{r-i}p^{pi} = \frac{r-s_r}{p-1} + r+ (p-1)i, \]
while the $p$-ordinal of this denominator equals
\[ {\rm ord}\:(pr)! = r + \frac{r-s_r}{p-1}. \]
Thus for $1\leq i\leq r$ expression (6.6) has $p$-ordinal greater than 0, hence (6.5) implies that $H(-r)$ is a $p$-adic unit for every positive integer $r$.  Lemma 6.2 follows because the negative integers are dense in ${\mathbb Z}_p$ and $H$ is continuous on ${\mathbb Z}_p$.
\end{proof}

{\bf Remark.}  The Dwork exponential function is defined by replacing $\pi_0$ by $\pi$, where $\pi^{p-1}=-p$, in the definition of $\sigma(t)$.  If we replaced $\pi_0$ by $\pi$ in the definition of $H(z)$, Equation (6.5) would hold with $u=0$.  
But from the functional equation for the $p$-adic gamma function $\Gamma_p$ we have $\Gamma_p(-pr) = r!(-p)^r/(pr)!$, so we would get $H(z) = \Gamma_p(pz)$.  This would lead to simpler formulas below, but the $p$-adic estimates arising from the use of $\pi$ in place of $\pi_0$ are not strong enough to prove Theorem~2.21.

 \section{The Dwork-Frobenius operator, part 1}

Since each $v_j$ is $p$-integral, for each $v_j$  there exists a unique rational number $v'_j$ in $[-1,0]$ such that $pv_j-v_j'\in\{0,-1,\dots,-(p-1)\}$.  Recall that if $v_j=-1$ then $v'_j=-1$ and if $v_j=0$ then $v'_j=0$.  Set $v' = (v'_1,\dots,v'_N)$.

We first construct for $-1<v_j\leq 0$ a Dwork-Frobenius map $\alpha:t^{v_j}C\{t\}\to t^{v'_j}C\{t\}$.
Let $\Phi:t^{v_j}C\{t\}\to t^{v'_j}C\{t\}$ be the map that replaces $t$ by $t^p$.  We denote by $\sigma(t):t^{v'_j}C\{t\}\to t^{v'_j}C\{t\}$ the map ``multiplication by $\sigma(t)$.''  Now define
\[ \alpha=\sigma(t)\circ\Phi:t^{v_j}C\{t\}\to t^{v'_j}C\{t\}. \]

Set 
\[ C^-\{t\} = \bigg\{\xi(t) = \sum_{l=-\infty}^0 c_lt^l\mid \xi(t)\in C\{t\}\bigg\}. \]
For $v_j=-1$ we construct a Dwork-Frobenius map $\alpha:t^{-1}C^-\{t\}\to t^{-1}C^-\{t\}$ by defining
\[ \alpha=\gamma\circ\sigma(t)\circ\Phi:t^{-1}C^-\{t\}\to t^{-1}C^-\{t\}. \]

Let $\xi'_j(t)$ be the series obtained by replacing $v_j$ by $v'_j$ in the definition of $\xi_j(t)$ (Equation (3.5)).  The following result is a variation of Boyarsky \cite{D1}.
\begin{theorem}
For $j=1,\dots,N$ we have
\[ \alpha(\xi_j(t)) = \pi_0^{-(p-1)v_j} \frac{H(-v_j)}{[v_j']_{pv_j-v_j'}}\xi'_j(t). \]
\end{theorem}

For $v_j=0$ we have $v'_j=0$ and the assertion reduces to $\alpha(\xi_j( t)) = \xi'_j(t)$ which is trivial from the definitions.
For $-1\leq v_j<0$ we need a lemma.  Let $D$ be the differential operator
\[ D = t\frac{d}{dt} - \pi_0 t, \]
which maps $t^{v_j}C\{t\}$ (resp.\ $t^{v'_j}C\{t\}$) to itself.
\begin{lemma}
{\bf (a).}  For $-1<v_j\leq 0$, we have $D\circ\alpha = p\alpha\circ D$ as maps from $t^{v_j}C\{t\}$ to $t^{v'_j}C\{t\}$. \\
{\bf (b).}  For $v_j=-1$ we have $(\gamma\circ D)\circ\alpha = p\alpha\circ(\gamma\circ D)$ as maps from
$t^{-1}C^-\{t\}$ to $t^{-1}C^-\{t\}$.
\end{lemma}

\begin{proof}
The operators $D$ and $\alpha$ are additive, i.~e., if $\xi(t) = \sum_{l\in {\mathbb Z}}c_l t^{v_j+l}\in t^{v_j}C\{t\}$, then
\[ D(\xi(t)) = \sum_{l\in {\mathbb Z}}c_l D(t^{v_j+l}) \]
and
\[ \alpha(\xi(t)) = \sum_{l\in {\mathbb Z}}c_l \alpha(t^{v_j+l}). \]
It thus suffices to verify the relation on monomials~$t^{v_j+l}$.

Suppose $-1<v_j\leq 0$.  On monomials, $\alpha$ factors formally as
\[ \alpha = \exp(\pi_0 t) \circ \Phi \circ \exp(-\pi_0 t) \]
and $D$ factors formally as
\[ D = \exp(\pi_0 t)\circ t\frac{d}{dt}\circ\exp(-\pi_0 t). \]
This reduces the assertion of part (a) to the obvious equality
\begin{equation}
 t\frac{d}{dt}\circ\Phi = p\Phi\circ t\frac{d}{dt}. 
 \end{equation}

If $I$ denotes the identity operator on $t^{-1}C\{t\}$, then it is straightforward to check that
$\gamma\circ D\circ(I-\gamma) = 0$ and $\gamma\circ(\sigma(t)\circ\Phi)\circ (I-\gamma) = 0$
as operators on $t^{-1}C\{t\}$.  These equalities imply that $\gamma\circ D\circ\gamma = \gamma\circ D$ and 
\[ \gamma\circ(\sigma(t)\circ\Phi)\circ\gamma = \gamma\circ(\sigma(t)\circ\Phi). \]
Using these relations, one reduces the assertion of part (b) to~(7.3).
\end{proof}

We determine the kernel of $D$ on $t^{v_j}C\{t\}$ for $-1<v_j\leq 0$.  Let 
\[ \xi(t) = \sum_{l=-\infty}^\infty c_l (\pi_0 t)^{v_j+l}\in t^{v_j}C\{t\}. \]
  Then
\[  D(\xi)= \sum_{l=-\infty}^\infty \big((v_j+l)c_l-c_{l-1}\big)(\pi_0 t)^{v_j+l}. \]
It follows that $D$ has a one-dimensional kernel generated by $\xi_j(t)$.

For $v_j=-1$ we determine the kernel of $\gamma\circ D$ on $t^{-1}C^-\{t\}$.  Let 
\[ \xi(t) = \sum_{l=-\infty}^0 c_l (\pi_0 t)^{-1+l}\in t^{-1}C^-\{t\}. \]
  Then
\[ \gamma\circ D(\xi) = \sum_{l=-\infty}^{0}\big((-1+l)c_l-c_{l-1}\big)(\pi_0 t)^{-1+l}. \]
It follows that $\gamma\circ D$ has a one-dimensional kernel again generated by $\xi_j(t)$.  

It follows from Lemma 7.2 that $\alpha$ maps the kernel of $D$ (resp.\ $\gamma\circ D$) in $t^{v_j}C\{t\}$ (resp.\ $t^{-1}C^-\{t\}$) to the kernel of $D$ (resp.\ $\gamma\circ D$) in $t^{v'_j}C\{t\}$ (resp.\ $t^{-1}C^-\{t\}$).  Since these kernels are all one-dimensional we have for all $j$
\begin{equation}
\alpha(\xi_j(t)) = \mu_j\xi'_j(t)
\end{equation}
for some $\mu_j\in {\mathbb C}_p$.  To finish the proof of Theorem~7.1 we need to compute $\mu_j$.  

Suppose first that $-1<v_j<0$.  To find $\mu_j$, we compute the coefficient of $t^{pv_j}$ on each side of~(7.4).  From the definition of $\xi'_j(t)$, the coefficient of $t^{pv_j}$ on the right-hand side of (7.4) is
\begin{equation}
\mu_j\pi_0^{pv_j}[v'_j]_{pv_j-v_j'}. 
\end{equation}
From the definitions, one computes that the coefficient of $t^{pv_j}$ on the left-hand side of (7.4) is
\begin{equation}
\pi_0^{v_j}\sum_{l=0}^\infty [v_j]_{-l}\sigma_{pl}\pi_0^{-l} = \pi_0^{v_j}H(-v_j),
\end{equation}
where $H(z)$ is defined in Section 6.  Comparing (7.5) and (7.6) then gives
\begin{equation}
\mu_j = \pi_0^{-(p-1)v_j}H(-v_j)/[v'_j]_{pv_j-v'_j},
\end{equation}
hence
\begin{equation}
\alpha(\xi_j(t)) = \pi_0^{-(p-1)v_j}\frac{H(-v_j)}{[v'_j]_{pv_j-v'_j}}\xi'_j(t).
\end{equation}
This proves Theorem 7.1 for $-1<v_j<0$.

Now suppose that $v_j=-1$, in which case $v'_j=-1$ also.  To compute $\mu_j$ we compute the coefficient of $t^{-p}$ on each side of (7.4).  For the right-hand side of (7.4) we take $l=(p-1)$ in the series expansion of $\xi'_j(t)$ (see (3.5)) to get
\[ \mu_j(-1)^{p-1}(p-1)!\pi_0^{-p}. \]
For the left-hand side of (7.4), a calculation from the definitions gives 
\[ \pi_0^{-1}\sum_{k=0}^\infty [-1]_{-k}\sigma_{pk}\pi_0^{-k} = \pi_0^{-1}H(1). \]
It follows that $\mu_j= (-\pi_0)^{p-1}H(1)/(p-1)!$, so
\begin{equation}
\alpha_j(\xi_j(t)) = (-\pi_0)^{p-1}\frac{H(1)}{(p-1)!}\xi'_j(t).
\end{equation}
Furthermore,
\[ (-\pi_0)^{p-1}\frac{H(1)}{(p-1)!} = \pi_0^{-(p-1)(-1)}\frac{H(1)}{[-1]_{-p+1}}, \]
proving Theorem 7.1 when $v_j=-1$.

\section{The Dwork-Frobenius operator, part 2}

The purpose of this section is to prove a result analogous to Theorem 7.1 for the~$\hat{\xi}_j(t)$.  For $-1<v_j\leq 0$, we define $\hat{\alpha}:t^{v_j}C\{t\}\to t^{v'_j}C\{t\}$ by $\hat{\alpha} = \theta(t)\circ\Phi$.  For $v_j=-1$ we define $\hat{\alpha}:t^{-1}C^-\{t\}\to t^{-1}C^-\{t\}$ by $\hat{\alpha} = \gamma\circ\theta(t)\circ\Phi$.  Let $\hat{\theta}_1(t): t^{v_j}C\{t\}\to t^{v_j}C\{t\}$ and $\hat{\theta}_1(t): t^{v'_j}C\{t\}\to t^{v'_j}C\{t\}$ be the maps ``multiplication by~$\hat{\theta}_1(t)$.''

\begin{proposition}
{\bf (a).}  We have $\hat{\alpha}\circ\hat{\theta}_1(t) = \hat{\theta}_1(t)\circ \alpha$ as maps from $t^{v_j}C\{t\}$ to~$t^{v'_j}C\{t\}$. \\
{\bf (b).}  We have $\hat{\alpha}\circ\big(\gamma\circ\hat{\theta}_1(t)\big) = \big(\gamma\circ\hat{\theta}_1(t)\big)\circ \alpha$ as maps from $t^{-1}C^-\{t\}$ to~$t^{-1}C^-\{t\}$.
\end{proposition}

\begin{proof}
All these operators are additive so it suffices to verify these relations on monomials.  On monomials, we can apply formal factorizations of these operators.  From the definitions
\begin{align*}
\hat{\alpha}\circ\hat{\theta}_1(t) &= \theta(t)\circ\Phi\circ\hat{\theta}_1(t) \\
 &=\hat{\theta}(t)\circ\Phi\circ\frac{1}{\hat{\theta}(t)}\circ\hat{\theta}_1(t) \\
  &= \hat{\theta}_1(t)\exp(\pi_0 t)\circ\Phi\circ\frac{1}{\exp(\pi_0 t)} \\
   &= \hat{\theta}_1(t)\sigma(t)\circ\Phi \\
    &= \hat{\theta}_1(t)\circ\alpha,
    \end{align*}
where the second equality follows from $\theta(t) = \hat{\theta}(t)/\hat{\theta}(t^p)$ and the third equality follows from
$\theta(t) = \hat{\theta}_1(t)\exp(\pi_0 t)$.  This proves part (a).  

From the definitions
\[ \hat{\alpha}\circ\gamma\circ\hat{\theta}_1(t) = \gamma\circ\theta(t)\circ\Phi\circ\gamma\circ\hat{\theta}_1(t). \]
We have $\gamma\circ\theta(t)\circ\Phi\circ(I-\gamma) = 0$ as operators on $t^{-1}C^-\{t\}$, so
\[  \hat{\alpha}\circ\gamma\circ\hat{\theta}_1(t) = \gamma\circ\theta(t)\circ\Phi\circ\hat{\theta}_1(t). \]
Similarly,
\[ \gamma\circ\hat{\theta}_1(t)\circ\gamma\circ\sigma(t)\circ\Phi = \gamma\circ\hat{\theta}_1(t)\sigma(t)\circ\Phi. \]
The assertion of part (b) now follows from part (a).
\end{proof}

Let $\hat{\xi}_j'(t)$ be defined by replacing $v_j$ by $v'_j$ in the definition of $\hat{\xi}_j(t)$.  Proposition~8.1 implies the analogue of Theorem 7.1.
\begin{theorem}
For $j=1,\dots,N$ we have
\[ \hat{\alpha}(\hat{\xi}_j(t)) = \pi_0^{-(p-1)v_j} \frac{H(-v_j)}{[v_j']_{pv_j-v_j'}}\hat{\xi}'_j(t). \]
\end{theorem}

\begin{proof}
Suppose first that $-1<v_j\leq 0$.  We have
\begin{align*}
\hat{\alpha}(\hat{\xi}_j(t)) &= \hat{\alpha}\big(\hat{\theta}_1(t)\xi_j(t)\big) \\
 &= \hat{\theta}_1(t)\alpha(\xi_j(t)) \\
  &= \hat{\theta}_1(t) \pi_0^{-(p-1)v_j} \frac{H(-v_j)}{[v_j']_{pv_j-v_j'}}\xi'_j(t) \\
   &= \pi_0^{-(p-1)v_j} \frac{H(-v_j)}{[v_j']_{pv_j-v_j'}}\hat{\xi}'_j(t),
   \end{align*}
where the first and fourth equalities follow from (3.10), the second equality follows from Proposition 8.1(a), and the third equality follows from Theorem~7.1. 

Now suppose that $v_j=-1$.  We have
\begin{align*}
\hat{\alpha}(\hat{\xi}_j(t)) &= \hat{\alpha}\big(\gamma(\hat{\theta}_1(t)\xi_j(t))\big) \\
 &= \gamma\big(\hat{\theta}_1(t)\alpha(\xi_j(t))\big) \\
  &= \gamma\big(\hat{\theta}_1(t) \pi_0^{-(p-1)v_j} \frac{H(-v_j)}{[v_j']_{pv_j-v_j'}}\xi'_j(t)\big) \\
   &= \pi_0^{-(p-1)v_j} \frac{H(-v_j)}{[v_j']_{pv_j-v_j'}}\hat{\xi}'_j(t),
   \end{align*}
where the first and fourth equalities follow from (3.14), the second equality follows from Proposition 8.1(b), and the third equality follows from Theorem~7.1. 
\end{proof}

To facilitate separation of variables in the next section we replace $t$ by $t_j$ in each~$\hat{\xi}_j$.  Let $\Phi_j$ be the operator that replaces $t_j$ by $t_j^p$ and define $\gamma_j$ to be the operator on series in $t_j$
\[ \gamma_j\bigg(\sum_{l\in{\mathbb Z}} c_lt_j^l\bigg) = \sum_{l=-\infty}^{-1} c_lt_j^l. \]
Replace $\hat{\alpha}$ by $\hat{\alpha}_j$, where $\hat{\alpha}_j = \theta(t_j)\circ\Phi_j$ for $-1<v_j\leq 0$ and $\hat{\alpha}_j = \gamma_j\circ\theta(t_j)\circ\Phi_j$ for $v_j=-1$.  Then Theorem 8.2 becomes
\begin{equation}
\hat{\alpha}_j(\hat{\xi}_j(t_j)) = \pi_0^{-(p-1)v_j} \frac{H(-v_j)}{[v_j']_{pv_j-v_j'}}\hat{\xi}'_j(t_j).
\end{equation}

\section{Dwork-Frobenius on hypergeometric series}

Put $\hat{\xi}_v(t_1,\dots,t_N) = \prod_{j=1}^N \hat{\xi}_j(t_j)$.  Then $\hat{\xi}_v(t_1,\dots,t_N)\in t^vC_E\{t_1,\dots,t_N\}$ by Lemma~5.2 ($C_E\{t_1,\dots,t_N\}$ is defined in (4.2)).  We also define
\[ \hat{\xi}_{v'}(t_1,\dots,t_N) = \prod_{j=1}^N \hat{\xi}'_j(t_j)\in t^{v'}C_E\{t_1,\dots,t_N\}. \]

Let $\Phi_N:t^vC\{t_1,\dots,t_N\}\to t^{v'}C\{t_1,\dots,t_N\}$, be the operator that replaces each $t_j$ by $t_j^p$.  Put $\theta(t_1,\dots,t_N) = \prod_{j=1}^N \theta(t_j)$.  Multiplication by $\theta(t_1,\dots,t_N)$ defines a map
\[ \theta(t_1,\dots,t_N):t^{v'}C_E\{t_1,\dots,t_N\}\to t^{v'}C_{E_+}\{t_1,\dots,t_N\}. \]
Let $\gamma_E:t^{v'}C\{t_1,\dots,t_N\}\to t^{v'}C_E\{t_1,\dots,t_N\}$ be defined by 
\[ \gamma_E\bigg(t^{v'}\sum_{l\in{\mathbb Z}^N} c_lt^l\bigg) = t^{v'}\sum_{l\in E}c_lt^l. \]
Finally we define $\hat{\alpha}_N:t^vC_E\{t_1,\dots,t_N\}\to t^{v'}C_E\{t_1,\dots,t_N\}$ by 
\[ \hat{\alpha}_N = \gamma_E\circ \theta(t_1,\dots,t_N)\circ\Phi_N. \]
By (4.15), (8.3), and a calculation we have
\begin{align}
\hat{\alpha}_N(\hat{\xi}_v(t_1,\dots,t_N)) &= \prod_{j=1}^N \hat{\alpha}_j(\hat{\xi}_j(t_j)) \\ \nonumber
 & = \bigg(\prod_{j=1}^N \pi_0^{-(p-1)v_j} \frac{H(-v_j)}{[v_j']_{pv_j-v_j'}}\bigg)\hat{\xi}_{v'}(t_1,\dots,t_N).
\end{align}
Since $\sum_{j=1}^N v_j{\bf a}_j = \beta$ and $w({\bf a}_j)=1$ for all $j$ we have $\sum_{j=1}^N v_j = w(\beta)$, so this equation simplifies to
\begin{equation}
\hat{\alpha}_N(\hat{\xi}_v(t_1,\dots,t_N)) = \pi_0^{-(p-1)w(\beta)} \bigg(\prod_{j=1}^N  \frac{H(-v_j)}{[v_j']_{pv_j-v_j'}}\bigg)\hat{\xi}_{v'}(t_1,\dots,t_N).
\end{equation}
Note that the factor $\prod_{j=1}^N  \frac{H(-v_j)}{[v_j']_{pv_j-v_j'}}$ is a $p$-adic unit.  The factors $H(-v_j)$ are $p$-adic units by Lemma 6.2 and the $[v_j']_{pv_j-v'_j}$ are $p$-adic units by a direct calculation.

From (3.18) we have
\begin{equation}
\hat{\xi}_v(\Lambda_1x^{{\bf a}_1},\dots,\Lambda_N x^{{\bf a}_N}) = \sum_{u\in\beta + {\mathbb Z}A} G_u(\Lambda)\pi_0^{w(u)}x^u. 
\end{equation}
We set
\begin{equation}
 G_v(\Lambda,x) =  \sum_{u\in\beta + {\mathbb Z}A} G_u(\Lambda)\pi_0^{w(u)}x^u. 
 \end{equation}
Put $\beta' = \sum_{j=1}^N v'_j{\bf a}_j$.  As in (3.18) we have
\begin{equation}
\prod_{j=1}^N \hat{\xi}'_j(\Lambda_jx^{{\bf a}_j}) = \sum_{u'\in\beta'+{\mathbb Z}A} G_{u'}(\Lambda)\pi_0^{w(u')}x^{u'},
\end{equation}
where
\begin{equation}
G_{u'}(\Lambda_1,\dots,\Lambda_N) = \sum_{l\in E(u')} [v']_l\Lambda^{v'+l}.
\end{equation}
As in (9.3) we have
\begin{equation}
\hat{\xi}_{v'}(\Lambda_1x^{{\bf a}_1},\dots,\Lambda_Nx^{{\bf a}_N}) = \sum_{u'\in\beta'+{\mathbb Z}A} G_{u'}(\Lambda)\pi_0^{w(u')}x^{u'},
\end{equation} 
and as in (9.4) we set
\begin{equation}
G_{v'}(\Lambda,x) = \sum_{u'\in\beta'+{\mathbb Z}A} G_{u'}(\Lambda)\pi_0^{w(u')}x^{u'}.
\end{equation}

For $S\subseteq{\mathbb Z}^N$ we set
\begin{multline}
C_S\{\Lambda_1x^{{\bf a}_1},\dots,\Lambda_Nx^{{\bf a}_N}\}\\ 
=\bigg\{\xi = \sum_{l=(l_1,\dots,l_N)\in S} c_l(\Lambda_1x^{{\bf a}_1})^{l_1}\cdots(\Lambda_Nx^{{\bf a}_N})^{l_N}\mid \text{$\lvert c_l\rvert\delta^{\lvert l\rvert_\infty}\to 0$ for all $\delta<1$}\bigg\}.
\end{multline}
The change of variable $t_j\to \Lambda_jx^{a_j}$ induces a ${\mathbb C}_p$-isomorphism
\begin{equation}
 t^vC_E\{t_1,\dots,t_N\}\to \Lambda^vx^\beta C_E\{\Lambda_1x^{{\bf a}_1},\dots,\Lambda_Nx^{{\bf a}_N}\}. 
 \end{equation}
 
 We let $\alpha^*:\Lambda^vx^\beta C_E\{\Lambda_1x^{{\bf a}_1},\dots,\Lambda_Nx^{{\bf a}_N}\}\to \Lambda^{v'}x^{\beta'}C_E\{\Lambda_1x^{{\bf a}_1},\dots,\Lambda_Nx^{{\bf a}_N}\}$ be the map corresponding to $\hat{\alpha}_N$ under this isomorphism.  Like $\hat{\alpha}_N$, the map $\alpha^*$ can be written as a composition.  Let 
\[ \Phi^*:\Lambda^vx^\beta C_E\{\Lambda_1x^{{\bf a}_1},\dots,\Lambda_Nx^{{\bf a}_N}\}\to \Lambda^{v'}x^{\beta'} C_E\{\Lambda_1x^{{\bf a}_1},\dots,\Lambda_Nx^{{\bf a}_N}\} \]
 be the map which replaces each $\Lambda_j$ by $\Lambda_j^p$ and each $x_i$ by $x_i^p$.  Put 
 \begin{equation}
 \theta(\Lambda,x) = \prod_{j=1}^N \theta(\Lambda_jx^{\bf a}_j).
 \end{equation}
 Multiplication by $\theta(\Lambda,x)$ defines a map
 \[ \theta(\Lambda,x):\Lambda^{v'}x^{\beta'} C_E\{\Lambda_1x^{{\bf a}_1},\dots,\Lambda_Nx^{{\bf a}_N}\}\to \Lambda^{v'}x^{\beta'} C_{E_+}\{\Lambda_1x^{{\bf a}_1},\dots,\Lambda_Nx^{{\bf a}_N}\}. \]
 Then
\begin{equation}
\alpha^* = \gamma_E\circ \theta(\Lambda,x)\circ\Phi^*,
\end{equation}
where $\gamma_E:\Lambda^{v'}x^{\beta'}C\{\Lambda_1x^{{\bf a}_1},\dots,\Lambda_Nx^{{\bf a}_N}\}\to \Lambda^{v'}x^{\beta'}C_E\{\Lambda_1x^{{\bf a}_1},\dots,\Lambda_Nx^{{\bf a}_N}\}$ is the map
\[ \gamma_E\bigg(\Lambda^{v'}x^{\beta'}\sum_{l\in {\mathbb Z}^N} c_l\prod_{j=1}^N (\Lambda_j x^{{\bf a}_j})^{l_j}\bigg) = \Lambda^{v'}x^{\beta'}\sum_{l\in E} c_l\prod_{j=1}^N (\Lambda_j x^{{\bf a}_j})^{l_j}. \]
Using the isomorphism (9.10), Equation (9.2) gives the following result.
\begin{theorem}
We have 
\begin{equation}
 \alpha^*(G_v(\Lambda,x)) = \pi_0^{-(p-1)w(\beta)}\bigg(\prod_{j=1}^N \frac{H(-v_j)}{[v_j']_{pv_j-v_j'}}\bigg) G_{v'}(\Lambda,x). 
 \end{equation}
\end{theorem}

\section{Restriction to ${\mathcal M}$}

In general, not all the hypergeometric series that appear in the generating series $G_v(\Lambda,x)$ will have $p$-integral coefficients.  We need to truncate $G_v(\Lambda,x)$ to include only those series that are potentially $p$-integral and we need a version of Theorem~9.13 for these truncations.  

Let
\[ \gamma_{\mathcal M}: \Lambda^vx^\beta C_{E_+}\{\Lambda_1x^{{\bf a}_1},\dots,\Lambda_Nx^{{\bf a}_N}\}\to \Lambda^vx^\beta C_E\{\Lambda_1x^{{\bf a}_1},\dots,\Lambda_Nx^{{\bf a}_N}\} \]
be the ${\mathbb C}_p$-linear map defined by
\[ \gamma_{\mathcal M}(\Lambda^{v+l}x^{\sum_{j=1}^N (v_j+l_j){\bf a}_j}) = \begin{cases} 
\Lambda^{v+l}x^{\sum_{j=1}^N (v_j+l_j){\bf a}_j} & \text{if $\sum_{j=1}^N (v_j+l_j){\bf a}_j\in{\mathcal M}$,} \\
0 & \text{if $\sum_{j=1}^N (v_j+l_j){\bf a}_j\not\in{\mathcal M}$.} \end{cases} \]
Note that the image of $\gamma_{\mathcal M}$ lies in $\Lambda^vx^\beta C_E\{\Lambda_1x^{{\bf a}_1},\dots,\Lambda_Nx^{{\bf a}_N}\}$ by Proposition~2.9.  Similarly we define
\[ \gamma_{{\mathcal M}'}: \Lambda^{v'}x^{\beta'} C_{E_+}\{\Lambda_1x^{{\bf a}_1},\dots,\Lambda_Nx^{{\bf a}_N}\}\to \Lambda^{v'}x^{\beta'} C_E\{\Lambda_1x^{{\bf a}_1},\dots,\Lambda_Nx^{{\bf a}_N}\}. \]

Put $G_{v}^{\mathcal M}(\Lambda,x) = \gamma_{\mathcal M} (G_v(\Lambda,x))$ and $G_{v'}^{{\mathcal M}'}(\Lambda,x) = \gamma_{{\mathcal M}'} (G_{v'}(\Lambda,x)$).  Thus (see (9.4) and (9.8))
\[ G_{v}^{\mathcal M}(\Lambda,x) = \sum_{u\in{\mathcal M}}G_u(\Lambda)\pi_0^{w(u)}x^u \]
and
\[ G_{v'}^{{\mathcal M}'}(\Lambda,x) = \sum_{u'\in{\mathcal M}'}G_{u'}(\Lambda)\pi_0^{w(u')}x^{u'}. \]
 
 Consider condition (2.20) for $i=1$, i.~e.,
 \begin{equation}
 w(\beta') = \max\{w(u')\mid u'\in{\mathcal M}'\}.
 \end{equation}
 
 \begin{lemma}
Suppose that $\beta'$ satisfies $(10.1)$.  Then $\gamma_{{\mathcal M}'}\circ(I-\gamma_E)=0$ on \\ 
$\Lambda^{v'}x^{\beta'} C_{E_+}\{\Lambda_1x^{{\bf a}_1},\dots,\Lambda_Nx^{{\bf a}_N}\}$.
\end{lemma}

\begin{proof}
Consider a monomial 
\[ \xi=\Lambda^{v'}x^{\beta'}(\Lambda_1x^{{\bf a}_1})^{l_1}\cdots (\Lambda_Nx^{{\bf a}_N})^{l_N} \]
with $l=(l_1,\dots,l_N)\in E_+$.  If $l\in E$ then $\gamma_E(\xi) = \xi$ so $(I-\gamma_E)(\xi) = 0$.  If $l\not\in E$ then $\gamma_E(\xi)=0$ so $(I-\gamma_E)(\xi) = \xi$.  But since $\beta'$ satisfies (10.1) and $l\not\in E$, Proposition~2.9 implies that ${\beta' + \sum_{j=1}^N l_j{\bf a}_j}\not\in{\mathcal M}'$, hence $\gamma_{{\mathcal M}'}(\xi) = 0$.
\end{proof}

Define
\begin{multline*}
 \alpha^*_{{\mathcal M}'} := 
 \gamma_{{\mathcal M}'}\circ\Theta(\Lambda,x)\circ\Phi^*:\Lambda^vx^\beta C_{E}\{\Lambda_1x^{{\bf a}_1},\dots,\Lambda_Nx^{{\bf a}_N}\}\\
 \to \Lambda^{v'}x^{\beta'} C_E\{\Lambda_1x^{{\bf a}_1},\dots,\Lambda_Nx^{{\bf a}_N}\}.
 \end{multline*}
\begin{corollary}
Suppose that $\beta'$ satisfies $(10.1)$.  Then 
\begin{equation}
\alpha_{{\mathcal M}'}^*(G_v(\Lambda,x)) = \pi_0^{-(p-1)w(\beta)}\bigg(\prod_{j=1}^N \frac{H(-v_j)}{[v_j']_{pv_j-v_j'}}\bigg)G_{v'}^{{\mathcal M}'}(\Lambda,x). 
\end{equation}
\end{corollary}

\begin{proof}
It follows from Lemma 10.2 that $\gamma_{{\mathcal M}'} = \gamma_{{\mathcal M}'}\circ\gamma_E$.  The assertion of the corollary then follows from Equation~(9.14) by applying $\gamma_{{\mathcal M}'}$ to both sides. 
\end{proof}

We can now prove the main result of this section.
\begin{theorem}
Suppose that $\beta'$ satisfies $(10.1)$.  Then
\begin{equation}
 \alpha^*_{{\mathcal M}'}\big(G_{v}^{\mathcal M}(\Lambda,x)\big) = \pi_0^{-(p-1)w(\beta)}\bigg(\prod_{j=1}^N \frac{H(-v_j)}{[v_j']_{pv_j-v_j'}}\bigg)G_{v'}^{{\mathcal M}'}(\Lambda,x). 
 \end{equation}
\end{theorem}

\begin{proof}
From (9.11) we have
\begin{equation}
\theta(\Lambda,x) = \sum_{\nu\in{\mathbb N}A}\theta_\nu(\Lambda)x^\nu
\end{equation}
with 
\begin{equation}
\theta_\nu(\Lambda) = \sum_{\substack{m\in{\mathbb N}^N\\ \sum_{j=1}^N m_j{\bf a}_j = \nu}}\theta_m\Lambda^m.
\end{equation}

The right-hand sides of (10.4) and (10.6) are identical, so to prove (10.6) we need to show the left-hand side of (10.4) equals the left-hand side of~(10.6).  Since $G_v^{{\mathcal M}}(\Lambda,x) = \gamma_{\mathcal M}(G_v(\Lambda,x)$ it suffices to show that $\alpha^*_{{\mathcal M}'}$ annihilates all monomials of the form $\Lambda^{v+l}x^u$ with $\sum_{j=1}^N (v_j+l_j){\bf a}_j = u\not\in{\mathcal M}$.  

Let $\Lambda^{v+l}x^u$ be such a monomial.  Then $\Phi^*(\Lambda^{v+l}x^u) = \Lambda^{p(v+l)}x^{pu}$, and since $u\not\in{\mathcal M}$ we have $pu\not\in{\mathcal M}'$.  It follows from (10.7) and (10.8) that the monomials in $\Theta(\Lambda,x)\Phi^*(\Lambda^{v+l}x^u)$ are of the form $\Lambda^{p(v+l) + m}x^{pu+\nu}$ where
$m\in{\mathbb N}^N$ and $\sum_{j=1}^N m_j{\bf a}_j = \nu$.  We thus have
\begin{equation}
 \sum_{j=1}^N (p(v_j+l_j)+m_j){\bf a}_j = pu+\nu. 
 \end{equation}
We need to show that $pu+\nu\not\in{\mathcal M}'$.

Note that $pl+m\in E_+$.  If $pl_j+m_j\neq 0$ for some $j$ with $-{\bf a}_j\not\in\sigma$, then Lemma~2.8 tells us that $p(v+l)+m\not\in{\mathcal M}'$ and we are done.  So suppose that $pl_j+m_j=0$ for all $j$ for which $-{\bf a}_j\not\in\sigma$.  We noted earlier that $-{\bf a}_j\not\in\sigma$ implies that $v_j=0$, hence $l_j\geq 0$ because $l\in E$.  But $m_j\geq 0$ for all $j$, so if $pl_j + m_j=0$ then $l_j=m_j=0$.  

It follows that if some $m_j>0$ in (10.9), then $-{\bf a}_j\in\sigma$, so 
\[ -\nu=-\sum_{j=1}^N m_j{\bf a}_j\in\sigma. \]
If $pu+\nu$ were in ${\mathcal M}'$, then adding this expression to (10.9) would give
\[ \sum_{j=1}^N p(v_j+l_j){\bf a}_j = pu\in{\mathcal M'}. \]
But this contradicts the assumption that $u\not\in{\mathcal M}$.
\end{proof}

\section{Proof of Theorem 2.21}

We suppose that our prime $p$ satisfies $p\equiv h\pmod{D}$ and that (2.20) holds for $i=0,\dots,a-1$.  
By Proposition 3.19 it suffices to prove that all the corresponding series $G_{u^{(i)}}(\Lambda)$ have $p$-integral coefficients.  By our assumption (2.20), Theorem 10.5 holds for all $\beta^{(i)}$ so we have
\begin{multline}
 \alpha^*_{{\mathcal M}^{(i+1)}}\big(G_{v^{(i)}}^{{\mathcal M}^{(i)}}(\Lambda,x)\big) \\
 = \pi_0^{-(p-1)w(\beta^{(i)})}\bigg(\prod_{j=1}^N \frac{H(-v_j^{(i)})}{[v_j^{(i+1)}]_{pv_j^{(i)}-v_j^{(i+1)}}}\bigg)G_{v^{(i+1)}}^{{\mathcal M}^{(i+1)}}(\Lambda,x). 
 \end{multline}
We begin by computing the left-hand side of (11.1) directly from the definition.  We have
\[ \Phi^*(G_{v^{(i)}}^{{\mathcal M}^{(i)}}(\Lambda,x)) = \sum_{u^{(i)}\in{\mathcal M}^{(i)}} G_{u^{(i)}}(\Lambda^p)\pi_0^{w(u^{(i)})}x^{pu^{(i)}}. \]
Thus from (10.7)
\[ \theta(\Lambda,x)\Phi^*(G_{v^{(i)}}^{{\mathcal M}^{(i)}}(\Lambda,x)) = \sum_{\rho\in E_+} \bigg( \sum_{\substack{\nu\in{\mathbb N}A,\,u^{(i)}\in{\mathcal M}^{(i)}\\ \nu+ pu^{(i)} = \rho}}\theta_\nu(\Lambda)G_{u^{(i)}}(\Lambda^p)\pi_0^{w(u^{(i)})}\bigg) x^\rho. \]
Finally, applying $\gamma_{{\mathcal M}^{(i+1)}}$, we get
\begin{multline}
\alpha_{{\mathcal M}^{(i+1)}}(G_{v^{(i)}}^{{\mathcal M}^{(i)}}(\Lambda,x)) \\
= \sum_{u^{(i+1)}\in{\mathcal M}^{(i+1)}}\bigg(\sum_{\substack{\nu\in{\mathbb N}A,\,u^{(i)}\in{\mathcal M}^{(i)}\\ \nu+ pu^{(i)} = u^{(i+1)}}}\theta_\nu(\Lambda)G_{u^{(i)}}(\Lambda^p)\pi_0^{w(u^{(i)}-u^{(i+1)})}\bigg) \pi_0^{w(u^{(i+1)})}x^{u^{(i+1)}}.
\end{multline}

Using the definition of $G_{v^{(i+1)}}^{{\mathcal M}^{(i+1)}}(\Lambda,x)$ the right-hand side of (11.1) equals
\begin{equation} 
\pi_0^{-(p-1)w(\beta^{(i)})}\bigg(\prod_{j=1}^N \frac{H(-v_j^{(i)})}{[v_j^{(i+1)}]_{pv_j^{(i)}-v_j^{(i+1)}}}\bigg)\sum_{u^{(i+1)}\in{\mathcal M}^{(i+1)}} G_{u^{(i+1)}}(\Lambda) \pi_0^{w(u^{(i+1)})}x^{u^{(i+1)}}.
\end{equation}
 By (11.1), the coefficients of $\pi_0^{w(u^{(i+1)})}x^{u^{(i+1)}}$ on the right-hand sides of (11.2) and in (11.3) must be equal for all $u^{(i+1)}\in{\mathcal M}^{(i+1)}$:
\begin{multline}
\pi_0^{-(p-1)w(\beta^{(i)})}\bigg(\prod_{j=1}^N \frac{H(-v_j^{(i)})}{[v_j^{(i+1)}]_{pv_j^{(i)}-v_j^{(i+1)}}}\bigg)G_{u^{(i+1)}}(\Lambda)\\
=\sum_{\substack{\nu\in{\mathbb N}A,\,u^{(i)}\in{\mathcal M}^{(i)}\\ \nu+ pu^{(i)} = u^{(i+1)}}}\theta_\nu(\Lambda)G_{u^{(i)}}(\Lambda^p)\pi_0^{w(u^{(i)}-u^{(i+1)})}.
\end{multline}
Solve this equation for $G_{u^{(i+1)}}(\Lambda)$:
\begin{multline}
G_{u^{(i+1)}}(\Lambda) \\
=\pi_0^{(p-1)w(\beta^{(i)})}\bigg(\prod_{j=1}^N \frac{H(-v_j^{(i)})}{[v_j^{(i+1)}]_{pv_j^{(i)}-v_j^{(i+1)}}}\bigg)^{-1}
\sum_{\substack{\nu\in{\mathbb N}A,\,u^{(i)}\in{\mathcal M}^{(i)}\\ \nu+ pu^{(i)} = u^{(i+1)}}}\theta_\nu(\Lambda)G_{u^{(i)}}(\Lambda^p)\pi_0^{w(u^{(i)}-u^{(i+1)})}.
\end{multline}

In Equation (10.8) we have for $m\in{\mathbb N}^N$ with $\sum_{j=1}^N m_j{\bf a}_j = \nu$ that $\theta_m = \prod_{j=1}^N \theta_{m_j}$.  Hence by (3.7)
 \[ {\rm ord}\: \theta_m = \sum_{j=1}^N {\rm ord}\: \theta_{m_j} \geq \frac{\sum_{j=1}^N m_j}{p-1} = \frac{w(\nu)}{p-1}. \]
We can thus write $\theta_\nu(\Lambda) = \pi_0^{w(\nu)}\tilde{\theta}_\nu(\Lambda)$ where $\tilde{\theta}_\nu(\Lambda)$ has $p$-integral coefficients and (11.5) can be rewritten as
\begin{multline}
G_{u^{(i+1)}}(\Lambda) = \pi_0^{(p-1)w(\beta^{(i)})}\bigg(\prod_{j=1}^N \frac{H(-v_j^{(i)})}{[v_j^{(i+1)}]_{pv_j^{(i)}-v_j^{(i+1)}}}\bigg)^{-1}\\
\cdot\sum_{\substack{\nu\in{\mathbb N}A,\,u^{(i)}\in{\mathcal M}^{(i)}\\ \nu+ pu^{(i)} = u^{(i+1)}}}\tilde{\theta}_\nu(\Lambda)G_{u^{(i)}}(\Lambda^p)\pi_0^{w(\nu+u^{(i)}-u^{(i+1)})}.
\end{multline}
Since $\nu+pu^{(i)}=u^{(i+1)}$ we have $\nu+u^{(i)}-u^{(i+1)} = -(p-1)u^{(i)}$ so (11.6) becomes
\begin{multline}
G_{u^{(i+1)}}(\Lambda) \\
= \bigg(\prod_{j=1}^N \frac{H(-v_j^{(i)})}{[v_j^{(i+1)}]_{pv_j^{(i)}-v_j^{(i+1)}}}\bigg)^{-1}
\sum_{\substack{\nu\in{\mathbb N}A,\,u^{(i)}\in{\mathcal M}^{(i)}\\ \nu+ pu^{(i)}= u^{(i+1)}}}\pi_0^{(p-1)w(\beta^{(i)}-u^{(i)})}\tilde{\theta}_\nu(\Lambda)G_{u^{(i)}}(\Lambda^p).
\end{multline}
In this equation, $\big(\prod_{j=1}^N \frac{H(-v_j^{(i)})}{[v_j^{(i+1)}]_{pv_j^{(i)}-v_j^{(i+1)}}}\big)^{-1}$ is a $p$-adic unit and all coefficients of the polynomial $\pi_0^{(p-1)w(\beta^{(i)}-u^{(i)})}\tilde{\theta}_\nu(\Lambda)$ are $p$-integral by (2.20)

We now apply Equation (11.7) to deduce $p$-integrality of the series coefficients of $G_{u^{(i)}}(\Lambda)$.  Recall the series expansion (see Equation~(3.17))
\begin{equation}
G_{u^{(i)}}(\Lambda)=  \sum_{l=(l_1,\dots,l_N)\in E(u^{(i)})} \bigg(\prod_{j=1}^N g(v^{(i)}_j,l_j)\bigg) \Lambda^{v^{(i)}+l},
\end{equation}
where the $g(v_j^{(i)},l_j)$ satisfy 
\begin{equation}
{\rm ord}\:g(v_j^{(i)},l_j) = {\rm ord}\:[v^{(i)}_j]_{l_j}
\end{equation}
 by (3.13) and (3.16).

Since $l\in E(u^{(i)})$, if $v^{(i)}_j=-1$ then $l_j\in{\mathbb Z}_{\leq 0}$ so $\prod_{v^{(i)}_j=-1} [v^{(i)}_j]_{l_j}$ is $p$-integral.   We can thus focus attention on the $[v^{(i)}_j]_{l_j}$ for $-1<v^{(i)}_j\leq 0$.  
We prove that all monomials $\Lambda^{v^{(i)}+l}$ in all $G_{u^{(i)}}(\Lambda)$, $u^{(i)}\in{\mathcal M}^{(i)}$, and $i=0,1,\dots,a-1$, have $p$-integral coefficients by induction on $d(v^{(i)}+l):=\sum_{-1<v^{(i)}_j\leq 0} \max\{0,v_j^{(i)}+l_j\}$.  If $d(v^{(i)}+l)=0$, then $l_j\leq 0$ for $-1<v^{(i)}_j\leq 0$.  In this case $[v_j^{(i)}]_{l_j}$ is $p$-integral from its definition, and by (11.9) the coefficient of $\Lambda^{v^{(i)}+l}$ in $G_{u^{(i)}}(\Lambda)$ is a unit times  $\prod_{j=1}^N[v_j^{(i)}]_{l_j}$, hence is $p$-integral.  This holds for $i=0,1,\dots,a-1$.   

Fix $i$ and consider a monomial $\Lambda^{v^{(i+1)}+l'}$ in $G_{u^{(i+1)}}(\Lambda)$ with $d(v^{(i+1)}+l')>0$.  By induction we may assume that all monomials $\Lambda^{v^{(k)}+l}$ in all $G_{u^{(k)}}(\Lambda)$, $u^{(k)}\in{\mathcal M}^{(k)}$, $k=0,1\dots,a-1$, with $d(v^{(k)}+l)<d(v^{(i+1)}+l')$ have $p$-integral coefficients.  From the right-hand side of (11.7) one gets a formula for the coefficient of $\Lambda^{v^{(i+1)}+l'}$ in~$G_{u^{(i+1)}}(\Lambda)$.  But only monomials $\Lambda^{v^{(i)}+l}$ of $G_{u^{(i)}}(\Lambda)$ with 
\[ d(v^{(i)}+l)\leq d(v^{(i+1)}+l')/p \]
 can contribute to this formula.  All these monomials have $p$-integral coefficients by the induction hypothesis, so (11.7) implies that the coefficient of $\Lambda^{v^{(i+1)}+l'}$ in $G_{u^{(i+1)}}(\Lambda)$ is $p$-integral.  Since $i\in\{0,1,\dots,a-1\}$, $u^{(i+1)}\in{\mathcal M}^{(i+1)}$, and $l'\in E(u^{(i+1)})$ were arbitrary, the proof of Theorem 2.21 is complete by induction.

\section{Examples}

Let $c_{jk}\in{\mathbb Z}$ for $1\leq j\leq n$ and $1\leq k\leq m$ and let
\[ C_j(s):=C_j(s_1,\dots,s_m)=\sum_{k=1}^m c_{jk}s_k. \]
To avoid trivial cases, we assume that no $C_j$ is identically zero.  We also assume that for each $k$ there are at least two values of $j$ such that $c_{jk}\neq 0$, i.~e., each $s_k$ appears in at least two $C_j$ with nonzero coefficient.  We also assume that 
\begin{equation}
\sum_{j=1}^n c_{jk}=1\quad\text{for $k=1,\dots,m$.}
\end{equation}
This condition will guarantee that the associated $A$-hypergeometric system is regular holonomic.

Many classical hypergeometric equations have a solution of the form (Dwork-Loeser \cite[Appendix]{DL})
\begin{equation}
\sum_{s_1,\dots,s_m=0}^\infty (\theta_1)_{C_1(s)}\cdots(\theta_n)_{C_n(s)}\frac{t_1^{s_1}\cdots t_m^{s_m}}{s_1!\cdots s_m!}.
\end{equation}
In this section we give a $p$-integrality criterion for such series when all $\theta_i$ are $p$-integral rational numbers in the interval $[0,1]$.  

We impose two additional conditions.
If $\theta_i = 1$, then $(\theta_i)_{C_i(s)}$ is undefined for $C_i(s)<0$, so we assume that all $c_{ik}$ are nonnegative when $\theta_i=1$.  If $\theta_i=0$, then $(\theta_i)_{C_i(s)} = 0$ for $C_i(s)>0$, so we assume that all $c_{ik}$ are nonpositive when $\theta_i=0$.  Under these assumptions, the series (12.2) can be written as 
\begin{equation}
\sum_{s_1,\dots,s_m = 0}^\infty \frac{\prod_{\theta_i=1} C_i(s)!}{\prod_{\theta_i=0}(-1)^{C_i(s)}(-C_i(s))!} \prod_{0<\theta_i<1} (\theta_i)_{C_i(s)} \frac{t_1^{s_1}\cdots t_m^{s_m}}{s_1!\cdots s_m!}.
\end{equation}

For the study of such series, the assumption above that each $s_k$ appears in at least two $C_i(s)$ is not restrictive.  If, for example, $s_1$ appears in only one $C_i(s)$ we can multiply the series (12.3) by $\frac{s_1!}{s_1!}$ and the assumption will be satisfied.  For instance, this would replace the series ${}_1F_0(\theta_1;t_1)$ by the series ${}_2F_1(\theta_1,1;1;t_1)$, which has identical coefficients,

We describe the $A$-hypergeometric system for which a solution $F_\beta(\Lambda)$ has coefficients identical (up to sign) to those of (12.2).
Let ${\bf a}_1,\dots,{\bf a}_n$ be the standard unit basis vectors in ${\mathbb R}^n$ and for $k=1,\dots,m$ let
\[ {\bf a}_{n+k} = (c_{1k},\dots,c_{nk}). \]
Our hypothesis that for a fixed $k$ at least two $c_{jk}$ are nonzero implies that ${\bf a}_1,\dots,{\bf a}_{n+m}$ are all distinct.  Put $N=n+m$ and let $A = \{{\bf a}_i\}_{i=1}^{N}\subseteq {\mathbb Z}^n$.   
Condition (12.1) implies that the elements of the set $A$ all lie on the hyperplane $\sum_{i=1}^n u_i =1$ in ${\mathbb R}^n$.

Let $\Theta = (\theta_1,\dots,\theta_n)$ be a sequence of $p$-integral rational numbers in the interval~$[0,1]$.  Take
\[ v = (-\theta_1,\dots,-\theta_n,0,\dots,0) \]
where $0$ is repeated $m$ times.  This is a sequence of $p$-integral rational numbers in the interval $[-1,0]$, and we take
\[ \beta = \sum_{i=1}^N v_i{\bf a}_i = (-\theta_1,\dots,-\theta_n). \]
We show that $\beta$ lies in $-C(A)^\circ$.  

Put $c_i = \sum_{k=1}^m c_{ik}$.  We begin with the relation
\begin{equation}
\sum_{i=1}^n c_i{\bf a}_i = \sum_{k=1}^m {\bf a}_{n+k}.
\end{equation}
We rewrite this as
\begin{equation}
\sum_{i:c_i>0} c_i{\bf a}_i = \sum_{i:c_i<0} (-c_i){\bf a}_i + \sum_{k=1}^m {\bf a}_{n+k}.
\end{equation}
Let $\sigma$ be the smallest closed face of $C(A)$ containing $A_1:=\{ {\bf a}_i\mid c_i>0\}$.  
All coefficients on both sides of (12.5) are positive, so the face $\sigma$ also contains the set $A_2:=\{ {\bf a}_i\mid c_i<0\}$ and the set $A_3:=\{{\bf a}_{n+k}\mid k=1,\dots,m\}$.  Our hypothesis implies that if $\theta_i=0$ then $c_i<0$, so all ${\bf a}_i$ with $\theta_i=0$ lie in $A_2$ and hence lie on $\sigma$.  The only elements of $A$ that may not lie on $\sigma$ are those in the set
$A_4:=\{{\bf a}_i\mid \theta_i>0 \text{ and }c_i=0\}$, so $\sigma\cup A_4$ contains the set $A$.  This implies that
\[ \sum_{{\bf a}_i\in A_1}{\bf a}_i + \sum_{{\bf a}_i\in A_4}{\bf a}_i \in C(A)^\circ. \]
Since $A_1$ and $A_4$ are subsets of $\{{\bf a}_i\mid \theta_i>0\}$, it follows that $\sum_{\theta_i>0} \theta_i{\bf a}_i\in C(A)^\circ$.  Equivalently, $\beta$ is an interior point of $-C(A)$.

From the definition of the set $E$ we have
\begin{multline}
E = \{ l=(l_1,\dots,l_N)\in{\mathbb Z}^N\mid \text{$l_j\leq 0$ if $\theta_j=1$,} \\
\text{$l_j\geq 0$ if $\theta_j=0$ or if $j=n+1\dots,n+m$}\}.
\end{multline}
For $u\in\beta + {\mathbb Z}A$ we have
\begin{equation}
 E(u) = \bigg\{ l=(l_1,\dots,l_N)\in E\mid \sum_{j=1}^N (v_j+l_j){\bf a}_j = u\bigg\} 
 \end{equation}
and
\begin{equation}
 F_u(\Lambda_1,\dots,\Lambda_N) = \sum_{l\in E(u)}[v]_l\Lambda^{v+l}. 
 \end{equation}

We take $u=\beta$ in these formulas to verify that the series $F_\beta(\Lambda)$ has the same coefficients (up to sign) as the series (12.2).  From (12.7) we have
\[ E(\beta) = \biggl\{ l=(l_1,\dots,l_N)\in E\mid \sum_{j=1}^N l_j{\bf a}_j = {\bf 0}\bigg\}. \]
We need to solve 
\begin{equation}
{\bf 0} =  \sum_{j=1}^N l_j{\bf a}_j = (l_1+C_1(l_{n+1},\dots,l_{n+m}),\dots,l_n+C_n(l_{n+1},\dots,l_{n+m}))
\end{equation}
with $l\in E$.  From (12.9) we must have for $i=1,\dots,n$
\begin{equation}
l_i=-C_i(l_{n+1},\dots,l_{n+m}).
\end{equation}
From (12.6) we must have $l_{n+k}\geq 0$ for $k=1,\dots,m$.  If $\theta_i=1$, then $C_i$ has nonnegative coefficients, so the $l_i$ given by (12.10) is nonpositive.  If $\theta_i=0$, then $C_i$ has nonpositive coefficients, so the $l_i$ given by (12.10) is nonnegative.  It follows that
\begin{equation}
E(\beta) = \bigg\{ (-C_1(s_1,\dots,s_m),\dots,-C_n(s_1,\dots,s_m),s_1,\dots,s_m)\mid s_1,\dots,s_m\in{\mathbb Z}_{\geq 0}\bigg\}.
\end{equation}

Equation (12.8) with $u=\beta$ gives
\begin{align}
F_\beta(\Lambda) &= \sum_{s_1,\dots,s_m=0}^\infty [v]_s\Lambda^{\beta+s} \\ \nonumber
 & = \sum_{s_1,\dots,s_m=0}^\infty \prod_{i=1}^n [-\theta_i]_{-C_i(s_1,\dots,s_m)}\prod_{k=1}^m [0]_{s_k}\Lambda^{\beta+s}\\ \nonumber
  & =  \Lambda^\beta\sum_{s_1,\dots,s_m=0}^\infty \prod_{i=1}^n (-1)^{C_i(s)}(\theta_i)_{C_i(s_1,\dots,s_m)} \prod_{i=1}^n \Lambda_i^{-C_i(s_1,\dots,s_m)} \frac{\prod_{k=1}^m \Lambda_{n+k}^{s_k}}{s_1!\cdots s_m!},
 \end{align}
 where the last equality follows from (1.3).  Thus the coefficients of $F_\beta(\Lambda)$ equal (up to sign) the corresponding coefficients of (12.2).

We can thus apply Theorem 2.21 to this $A$-hypergeometric system to get an integrality condition for the classical hypergeometric series (12.2) and all the related series $F_u(\Lambda)$ with $u\in{\mathcal M}$.  
Fix a positive integer $D$ such that $D\theta_j\in{\mathbb Z}$ for all $j$ and fix a positive integer $h$ with $(h,D) = 1$.  
In Section 2, we used $h$ to define a map $v_j\to v_j'$ for $-1\leq v_j\leq 0$.  We denoted the $i$-fold iteration of this map by $v_j\to v_j^{(i)}$, and, if $h^a\equiv 1\pmod{D}$, then $v_j^{(a)} = v_j$.  We define $\theta_j^{(i)}$ by the formula
\[  \theta_j^{(i)} = -(-\theta_j)^{(i)}. \]
The $\theta_j^{(i)}$ lie in the interval $[0,1]$.  We set $\Theta^{(i)} = (\theta_1^{(i)},\dots,\theta_n^{(i)})$ and
\[ v^{(i)} = (-\theta_1^{(i)},\dots,-\theta_n^{(i)},0,\dots,0), \]
where $0$ is repeated $m$ times.

We have $\beta^{(i)} = \sum_{j=1}^N v_j^{(i)}{\bf a}_j$ and ${\mathcal M}^{(i)} = (\beta^{(i)}+{\mathbb Z}A)\cap(-C(A)^\circ)$.   For $u^{(i)}\in{\mathcal M}^{(i)}$ we have
\begin{equation}
 F_{u^{(i)}}(\Lambda_1,\dots,\Lambda_N) = \sum_{l\in E(u^{(i)})}[v^{(i)}]_l\Lambda^{v^{(i)}+l}. 
 \end{equation}
 Note that
 \begin{equation}
 w(\beta^{(i)}) = -\sum_{j=1}^n \theta^{(i)}_j.
 \end{equation}
 
Define a step function on ${\mathbb R}^m$ 
\[ \rho(\Theta;x_1,\dots,x_m) = \sum_{j=1}^n \lfloor 1-\theta_j + C_j(x_1,\dots,x_m)\rfloor. \]
\begin{theorem}
Suppose that for each $i=0,1,\dots,a-1$ we have
\[ \rho(\Theta^{(i)};x_1,\dots,x_m)\geq 0 \]
for all $x_1,\dots,x_m\in[0,1)$.  Then for $i=0,1,\dots,a-1$ and all $u^{(i)}\in{\mathcal M}^{(i)}$ the series $F_{u^{(i)}}(\Lambda)$ has $p$-integral coefficients for all primes $p\equiv h\pmod{D}$.
\end{theorem}

{\bf Remark.}  The special case of Theorem 12.15 where each $C_i$ has either all coefficients nonnegative or all coefficients nonpositive and where $u^{(i)} = \beta^{(i)}$ is \cite[Theorem 5.6]{AS}.  Theorem 12.15 extends that result to all $u^{(i)}$ in all ${\mathcal M}^{(i)}$ and allows each $C_i$, when $\theta_i\neq 0,1$, to have both positive and negative coefficients.  The special case of Theorem 12.15 where each $\theta_i$ equals either $0$ or $1$ and where each $C_i$ has either all coefficients nonnegative or all coefficients nonpositive is \cite[Theorem~2.6]{AS3}.  

Theorem 12.15 is an immediate consequence of Theorem 2.21 and the following result.
\begin{lemma}
Let $i\in\{0,1,\dots,a-1\}$.  Equation $(2.20)$ holds if and only if 
\[ \rho(\Theta^{(i)};x_1,\dots,x_m)\geq 0 \]
for all $x_1,\dots,x_m\in[0,1)$. 
\end{lemma}

\begin{proof}
It suffices to prove this in the case $i=0$ as the other cases are similar.   Since $w(\beta) = -\sum_{i=1}^n \theta_i$, We need to show that for all $u\in{\mathcal M}$ we have
\begin{equation}
w(u)\leq -\sum_{i=1}^n \theta_i
\end{equation}
if and only if 
\begin{equation}
\sum_{i=1}^n \lfloor 1-\theta_i+C_i(x_1,\dots,x_m)\rfloor \geq 0
\end{equation}
for all $x_1,\dots,x_m\in[0,1)$.  

Fix $\hat{u}\in{\mathcal M}$ such that $w(\hat{u}) = \max\{w(u)\mid u\in{\mathcal M}\}$ and choose $\tilde{u}=(\tilde{u}_1,\dots,\tilde{u}_n)\in{\mathbb Z}^n$ such that $\hat{u}=\beta + \tilde{u}$.  Since $w(\hat{u}) = w(\beta) + w(\tilde{u}) = -\sum_{i=1}^n \theta_i + \sum_{i=1}^n \tilde{u}_i$, we need to show that 
\begin{equation}
\sum_{i=1}^n \tilde{u}_i\leq 0
\end{equation}
 if and only if (12.18) holds for all $x_1,\dots,x_m\in[0,1)$.  
 
Since $\hat{u}$ is an interior point of $-C(A)$, by \cite[Lemma 1]{AS2} we may write
\begin{equation}
\hat{u} = \sum_{i=1}^N z_i{\bf a}_i
\end{equation}
with $z_i\leq 0$ for all $i$ and $z_i<0$ for $i=1,\dots,n$.  Note that since the coordinates of each ${\bf a}_i$ sum to $1$, we have
\begin{equation}
w(\hat{u}) = \sum_{i=1}^N z_i.
\end{equation}

We must have $z_i\geq -1$ for all $i$.  For if some $z_{i_0}<-1$, then
\begin{equation}
\hat{u}+{\bf a}_{i_0} = (z_{i_0}+1){\bf a}_{i_0} + \sum_{\substack{i=1\\ i\neq i_0}}^N z_i{\bf a}_i
\end{equation}
is an element of ${\mathcal M}$ since all coefficients on the right-hand side are less than or equal to $0$ and every ${\bf a}_i$ that occurs with a negative coefficient in (12.20) also occurs with a negative coefficient in (12.22).  But $w(\hat{u}+{\bf a}_{i_0})>w(\hat{u})$, contradicting the choice of $\hat{u}$.  

We claim that $z_i>-1$ for $i=n+1,\dots,n+m\ (=N)$.  If $z_{i_0}=-1$ for some $i_0\in\{n+1,\dots,N\}$, then (12.22) becomes 
\[ \hat{u}+{\bf a}_{i_0} = \sum_{\substack{i=1\\ i\neq i_0}}^N z_i{\bf a}_i  \]
But since $z_i<0$ for $i=1,\dots,n$, the point $\hat{u}+{\bf a}_{i_0}$ is an element of ${\mathcal M}$, and again $w(\hat{u}+{\bf a}_{i_0})>w(\hat{u})$, contradicting the choice of $\hat{u}$.  

In summary, we have proved that in the representation (12.20) one has
\begin{align}
z_i\in[-1,0)\quad & \text{for $i=1,\dots,n$,}\\
z_i\in(-1,0]\quad & \text{for $i=n+1,\dots,N$.}
\end{align}

We now examine (12.20) coordinatewise.  For $i=1,\dots,n$ we have 
\begin{equation}
\hat{u}_i = -\theta_i + \tilde{u}_i = z_i + C_i(z_{n+1},\dots,z_{n+m}).
\end{equation}
This equation shows that 
\[ z_i\equiv -\theta_i-C_i(z_{n+1},\dots,z_{n+m})\pmod{\mathbb Z}, \]
and since $-1\leq z_i<0$ this implies that
\begin{equation}
z_i = -1-\theta_i-C_i(z_{n+1},\dots,z_{n+m})-\lfloor -\theta_i-C_i(z_{n+1},\dots,z_{n+m})\rfloor.
\end{equation}
This implies by (12.25) that 
\begin{equation}
\tilde{u}_i = -1-\lfloor -\theta_i-C_i(z_{n+1},\dots,z_{n+m})\rfloor.
\end{equation}

It follows that
\begin{align}
\sum_{i=1}^n \tilde{u}_i &= -\sum_{i=1}^n \lfloor 1-\theta_i - C_i(z_{n+1},\dots,z_{n+m})\rfloor \\
 & = -\sum_{i=1}^n \lfloor 1-\theta_i + C_i(-z_{n+1},\dots,-z_{n+m})\rfloor. \nonumber
\end{align}
It is clear from this equation that if (12.18) holds for all $x_1,\dots,x_m\in[0,1)$ then (12.19) holds.  

Conversely, suppose there exist $x_1,\dots,x_m\in[0,1)$ such that
\begin{equation}
 \sum_{i=1}^n \lfloor 1-\theta_i + C_i(x_1,\dots,x_m)\rfloor <0. 
 \end{equation}
For $i=1,\dots,m$, define $z_{n+i} = -x_i$, so $-1<z_{n+i}\leq 0$.  Then for $i=1,\dots,n$ define $z_i$ by (12.26) and define $\tilde{u}_i$ by (12.27).  Then $-1\leq z_i<0$ for $i=1,\dots,n$, so (12.20) shows that $\hat{u}$ is an interior point of $-C(A)$.  Equation~(12.28) holds by the definition of the $\tilde{u}_i$.  Equations~(12.28) and~(12.29) then imply that $\sum_{i=1}^n \tilde{u}_i>0$.  This shows that the failure of (12.18) implies the failure of (12.19), which in turn implies that $w(\hat{u})>w(\beta)$, i.e., Equation (2.3) fails.
\end{proof}

As an application of Theorem 12.15, consider the two-variable Horn series
\begin{equation}
 G_1(\theta_1,\theta_2,\theta_3;t_1,t_2) = \sum_{s_1,s_2 = 0}^\infty (\theta_1)_{s_1+s_2} (\theta_2)_{s_2-s_1}(\theta_3)_{s_1-s_2}\frac{t_1^{s_1}t_2^{s_2}}{s_1!s_2!}
 \end{equation}
 with $0<\theta_1,\theta_2,\theta_3<1$.
As above, we let ${\bf a}_i$ for $i=1,2,3$ be the standard unit basis vectors for ${\mathbb R}^3$ and we take ${\bf a}_4=(1,-1,1)$ (the coefficients of $s_1$ in the subscripts for the $\theta_i$) and ${\bf a}_5 = (1,1,-1)$ (the coefficients of~$s_2$ in the subscripts for the $\theta_i$).  We have $C_1(x_1,x_2) = x_1+x_2$, $C_2(x_1,x_2) = x_2-x_1$, and $C_3(x_1,x_2) = x_1-x_2$.  To apply Theorem~12.15, we need to consider expressions of the form
\[ \lfloor 1-\theta_1+x_1+x_2\rfloor + \lfloor 1-\theta_2 + x_2-x_1\rfloor + \lfloor 1-\theta_3 + x_1-x_2\rfloor \]
and determine when they are nonnegative for all $x_1,x_2\in[0,1)$.  
\begin{lemma}
We have 
\begin{equation}
\lfloor 1-\theta_1+x_1+x_2\rfloor + \lfloor 1-\theta_2 + x_2-x_1\rfloor + \lfloor 1-\theta_3 + x_1-x_2\rfloor\geq 0
\end{equation}
for all $x_1,x_2\in[0,1)$ if and only if either $\theta_2+\theta_3\leq 1$ or both $\theta_1+\theta_2\leq 1$ and $\theta_1 + \theta_3\leq 1$.
\end{lemma}

\begin{proof}
On the left-hand side of (12.32), the first term is always nonnegative and the last two terms are greater than or equal to~$-1$.  Since
\[ (1-\theta_2 + x_2-x_1) + (1-\theta_3 + x_1-x_2) = 2-\theta_2-\theta_3\geq 0, \]
at most one of the last two terms on the left-hand side of (12.32) can be negative.  Thus the only way the left-hand side of (12.32) can be negative is if either the first two terms equal $0$ and the third term equals $-1$ or the first and third terms equal $0$ and the second term equal $-1$.

If $\theta_2+\theta_3\leq 1$, then
\[ (1-\theta_2 + x_2-x_1) + (1-\theta_3 + x_1-x_2) \geq 1, \]
which implies that the sum of the last two terms on the left-hand side of (12.32) is nonnegative, hence (12.32) holds for all $x_1,x_2\in[0,1)$.

Suppose that $\lfloor 1-\theta_2 + x_2-x_1\rfloor = -1$ and $\lfloor 1-\theta_3 + x_1-x_2\rfloor = 0$.  If $\theta_1 + \theta_2\leq 1$ then 
\[ (1-\theta_1 + x_1 + x_2) + (1-\theta_2 + x_2-x_1) = 2-\theta_1-\theta_2 + 2x_2\geq 1, \]
which implies that the sum of the first two terms on the left-hand side of (12.32) is nonnegative.  A similar argument applies to the case
$\lfloor 1-\theta_3 + x_1-x_2\rfloor = -1$ and $\lfloor 1-\theta_2 + x_2-x_1\rfloor = 0$ if $\theta_1 + \theta_3\leq 1$.  

To prove the converse, suppose that $\theta_2+\theta_3>1$ and $\theta_1 + \theta_2>1$.  Then
\begin{equation} \theta_1>1-\theta_2\text{ and } \theta_3>1-\theta_2.
\end{equation}
Take $x_2=0$ and choose $x_1$ so that $\theta_1,\theta_3>x_1>1-\theta_2$.  Then the first and third terms on the left-hand side of (12.32) equal 0 while the second term on the left-hand side of (12.32) equals $-1$.  A similar argument applies when $\theta_2 + \theta_3>1$ and $\theta_1 + \theta_3>1$.  
\end{proof}

From Lemma 12.31 and Theorem 12.15 we get the following.
\begin{corollary}
Suppose for each $i=0,1,\dots,a-1$ we have either $\theta_2^{(i)} + \theta_3^{(i)}\leq 1$ or both $\theta_1^{(i)} + \theta_2^{(i)}\leq 1$ and $\theta_1^{(i)}+\theta_3^{(i)}\leq 1$.  Then the series $G_1(\theta^{(i)}_1,\theta_2^{(i)},\theta_3^{(i)};t_1,t_2)$ have $p$-integral coefficients for all primes $p\equiv h\pmod{D}$.
\end{corollary}

By Equation (2.12) the series $F_{\beta^{(i)}}(\Lambda)$, with $\beta = (-\theta_1,-\theta_2,-\theta_3)$ and $\{{\bf a}_i\}_{i=1}^5$ as in this example, have the same coefficients (up to sign) as the $G_1(\theta^{(i)}_1,\theta_2^{(i)},\theta_3^{(i)};t_1,t_2)$.  
By Theorem 12.15, the hypothesis of Corollary 12.34 implies $p$-integrality for not just one series but for all the series $F_{u^{(i)}}(\Lambda)$ with $u^{(i)}\in{\mathcal M}^{(i)}$.  We describe these series.  To compute the series $F_{u^{(i)}}(\Lambda)$ given by (12.8) we need to find the set $E(u^{(i)})$ of (12.9).  For $u^{(i)}\in \beta^{(i)} + {\mathbb Z}A$ write
\[ u^{(i)} = \beta^{(i)} + \tilde{u}^{(i)} \]
with $\tilde{u}^{(i)} = (\tilde{u}^{(i)}_1,\tilde{u}^{(i)}_2,\tilde{u}^{(i)}_3)\in{\mathbb Z}^3$.  To find $E(u^{(i)})$ we need to solve the equation
\begin{equation}
\sum_{k=1}^3 (-\theta_k^{(i)}+l_k){\bf a}_i + l_4{\bf a}_4 + l_5{\bf a}_5 = (-\theta^{(i)}_1+\tilde{u}^{(i)}_1,-\theta^{(i)}_2+\tilde{u}^{(i)}_2,-\theta^{(i)}_3+\tilde{u}^{(i)}_3), 
 \end{equation}
subject to the condition that $(l_1,\dots,l_5)\in E$, i.~e., $l_k\in{\mathbb Z}$ for $k=1,2,3$ and $l_k\in{\mathbb Z}_{\geq 0}$ for $k=4,5$.  

Equation (12.35) simplifies to the system 
\[ l_1+l_4+l_5 = \tilde{u}^{(i)}_1,\quad l_2-l_4+l_5=\tilde{u}_2^{(i)},\quad l_3 +l_4-l_5=\tilde{u}_3^{(i)}. \]
The solution in $E$ is to take $l_4,l_5\in{\mathbb Z}_{\geq 0}$ arbitrary, then $l_1 = -l_4-l_5+\tilde{u}^{(i)}_1$, $l_2 = l_4-l_5+\tilde{u}^{(i)}_2$, and $l_3 = -l_4+l_5 + \tilde{u}^{(i)}_3$.  Substituting these values into (12.8) and using the relation (1.3) shows that the series
\begin{equation}
 \sum_{s_1,s_2=0}^\infty (\theta^{(i)}_1)_{s_1+s_2-\tilde{u}^{(i)}_1}(\theta_2^{(i)})_{s_2-s_1-\tilde{u}^{(i)}_2}(\theta_3^{(i)})_{s_1-s_2-\tilde{u}^{(i)}_3}\frac{t_1^{s_1}t_2^{s_2}}{s_1!s_2!} 
 \end{equation}
has $p$-integral coefficients for all primes $p\equiv h\pmod{D}$ when $u^{(i)}\in{\mathcal M}^{(i)}$ and the hypothesis of Corollary~12.34 is satisfied.

Choosing $u=\beta$ gives $\tilde{u} = {\bf 0}$ and (12.36)  becomes the series $G_1(\theta^{(i)}_1,\theta_2^{(i)},\theta_3^{(i)};x_1,x_2)$.  We give another example of $u\in{\mathcal M}$.  
The faces of the cone $C(A)$ lie in the planes $x_1=0$, $x_2+x_3=0$, $x_1+x_2=0$, and $x_1+x_3=0$.  The interior points of $-C(A)$ are the points where these four linear forms take on negative values.  Suppose that $\theta_2^{(i)} + \theta_3^{(i)}\leq 1$ for $i=0,1,\dots,a-1$, so the hypothesis of Corollary~12.34 is satisfied, but that $\theta_1+\theta_2>1$.  Then $u=(-\theta_1,-\theta_2+1,-\theta_3-1)$ is an interior point of $-C(A)$, so the series corresponding to $u$ also has $p$-integral coefficients.  Thus $\tilde{u} = (0,1,-1)$ and substitution into (12.36) gives
\[ \sum_{s_1,s_2=0}^\infty (\theta_1)_{s_1+s_2}(\theta_2)_{s_2-s_1-1}(\theta_3)_{s_1-s_2+1}\frac{t_1^{s_1}t_2^{s_2}}{s_1!s_2!}, \]
a series with $p$-integral coefficients for all primes $p\equiv h\pmod{D}$.


\begin{thebibliography}{99}

\bibitem{AS1} A. Adolphson and S. Sperber.  Distinguished-root formulas for generalized Calabi-Yau hypersurfaces. Algebra Number Theory {\bf 11} (2017), no.\ 6, 1317--1356.

\bibitem{AS2} A. Adolphson and S. Sperber.  On the integrality of factorial ratios and mirror maps.  Integers {\bf 20} (2020), art.\ A10, 15pp.

\bibitem{AS3} A. Adolphson and S. Sperber.  On the integrality of hypergeometric series whose coefficients are factorial ratios.  Acta Arith.\ {\bf 200} (2021), no.\ 1, 39--59.

\bibitem{AS} A. Adolphson and S. Sperber.  On integrality properties of hypergeometric series.  Funct.\ Approx.\ Comment.\ Math.\ {\bf 65} (2021), no.\ 1, 7--31. 

\bibitem{AS4} A. Adolphson and S. Sperber. $A$-hypergeometric series and a $p$-adic refinement of the Hasse-Witt matrix.  
Abh.\ Math.\ Semin.\ Univ.\ Hambg.\ {\bf 91} (2021), no.\ 2, 225--256.

\bibitem{D1} M. Boyarsky.  $p$-adic gamma functions and Dwork cohomology. Trans.\ Amer.\ Math.\ Soc.\ {\bf 257} (1980), no.\ 2, 359--369.

\bibitem{D} B. Dwork.  $p$-adic cycles. Inst.\ Hautes \'Etudes Sci.\ Publ.\ Math.\ No.\ 37, (1969), 27--115.

\bibitem{DL} B. Dwork and F. Loeser.  Hypergeometric series.  Japan.\ J. Math.\ (N. S.) {\bf 19} (1993), no.\ 1, 81--129.

\end{thebibliography}
\end{document}